\documentclass[11pt,a4paper,reqno]{amsart}

\usepackage{amssymb,mathrsfs,bbm}
\usepackage{times}
\usepackage{fullpage}
\usepackage[latin1]{inputenc}
\usepackage{hyperref}

\newcommand{\diff}{\mt{d}}
\newcommand{\mt}[1]{{\text{\rm #1}}}  
\newcommand{\comment}[1]{}            

\newcommand{\set}[1]{\{#1\}}
\newcommand{\bigset}[1]{\big\{#1\big\}}
\newcommand{\Bigset}[1]{\Big\{#1\Big\}}

\newcommand{\Biggset}[1]{\Bigg\{#1\Bigg\}}
\newcommand{\suchthat}{\;|\;}         
\newcommand{\bigsuchthat}{\;\big|\;}
\newcommand{\Bigsuchthat}{\;\Big|\;}

\newcommand{\restrict}{\,|}           
\newcommand{\compose}{\circ}          
\newcommand{\define}{\mathrel{\rm:=}}

\newcommand{\without}{\mathord{\setminus}}

\newcommand{\R}{\mathbb{R}}           
\newcommand{\N}{\mathbb{N}}           
\newcommand{\leer}{\varnothing}       
\newcommand{\mfbd}{\partial}          
\newcommand{\RP}{\mathbb{RP}}         
\newcommand{\bigeval}[2]{\big\langle #1,#2 \big\rangle}
\newcommand{\Bigeval}[2]{\Big\langle #1,#2 \Big\rangle}

\newcommand{\abs}[1]{\lvert#1\rvert}  
\newcommand{\bigabs}[1]{\big\lvert#1\big\rvert}
\newcommand{\Bigabs}[1]{\Big\lvert#1\Big\rvert}
\newcommand{\biggabs}[1]{\bigg\lvert#1\bigg\rvert}
\newcommand{\Biggabs}[1]{\Bigg\lvert#1\Bigg\rvert}
\newcommand{\norm}[1]{\lVert#1\rVert} 
\newcommand{\bignorm}[1]{\big\lVert#1\big\rVert}
\newcommand{\Bignorm}[1]{\Big\lVert#1\Big\rVert}

\newcommand{\eps}{\varepsilon}        
\newcommand{\ooi}[2]{\mathord{\left]#1,#2\right[}}  
\newcommand{\oci}[2]{\left]#1,#2\right]}  
\newcommand{\coi}[2]{\left[#1,#2\right[}  
\newcommand{\cci}[2]{\left[#1,#2\right]}  

\newcommand{\supp}{\mt{supp}}

\swapnumbers
\theoremstyle{definition}
\newtheorem{remark}{Remark}[section]
\newtheorem{remarks}[remark]{Remarks}
\newtheorem{void}[remark]{}

\newtheorem{definition}[remark]{Definition}

\newtheorem{facts}[remark]{Facts}

\newtheorem{notations}[remark]{Notations}

\newtheorem{examples}[remark]{Examples}

\newtheorem{conventions}[remark]{Conventions}

\newtheorem{exhaustions}[remark]{Compact exhaustions}
\newtheorem{semicont}[remark]{Upper and lower semicontinuity}
\newtheorem{yamabe}[remark]{Yamabe constant and $\sigma$-invariant}
\theoremstyle{plain}
\newtheorem{theorem}[remark]{Theorem}
\newtheorem{lemma}[remark]{Lemma}
\newtheorem{factit}[remark]{Fact}
\newtheorem{factsit}[remark]{Facts}
\newtheorem{exampleit}[remark]{Example}
\newtheorem{corollary}[remark]{Corollary}

\newtheorem{conjecture}[remark]{Conjecture}

\newcommand{\Ric}{\mathord{\mt{Ric}}}

\newcommand{\Sym}{\mathord{\mt{Sym}}}
\newcommand{\divergence}{\mathord{\mt{div}}}

\newcommand{\GL}{\mt{GL}}

\renewcommand{\bar}{\overline}

\newcommand{\laplace}{\mathord{\Delta}}
\newcommand{\scal}{\mt{scal}}
\newcommand{\pr}{\mt{pr}}
\newcommand{\BerardBergery}{B\'{e}rard Bergery}
\newcommand{\Yam}{Y}
\newcommand{\Yamfunc}{\mt{$E$}}
\newcommand{\Yaminfty}{\bar{Y}}
\DeclareMathOperator{\Metr}{Metr}
\DeclareMathOperator{\vol}{vol}
\newcommand{\graph}{\mt{graph}}

\newcommand{\F}[2]{\tfrac{\diff\mu_{#1}}{\diff\mu_{#2}}}
\newcommand{\Tamma}{\tilde{\Gamma}}

\begin{document}

\title{The Yamabe constant on noncompact manifolds}

\author{Nadine Gro{\ss}e}
\address{Department of Mathematics, University of Leipzig}
\email{Nadine.Grosse@math.uni-leipzig.de}

\author{Marc Nardmann}
\address{Department of Mathematics, University of Hamburg}
\email{Marc.Nardmann@math.uni-hamburg.de}

\thanks{This work was done during our joint stay at the Hausdorff Institute for Mathematics. We thank the organizers for the hospitality.}

\begin{abstract}
We prove several facts about the Yamabe constant of Riemannian metrics on general noncompact manifolds and about S.~Kim's closely related ``Yamabe constant at infinity''. In particular we show that the Yamabe constant depends continuously on the Riemannian metric with respect to the fine $C^2$-topology, and that the Yamabe constant at infinity is even locally constant with respect to this topology. We also discuss to which extent the Yamabe constant is continuous with respect to coarser topologies on the space of Riemannian metrics.
\end{abstract}

\maketitle

\section{Introduction}

For a nonempty manifold $M$ of dimension $n\geq3$, the \emph{Yamabe map} \comment{$\Yam=$}$\Yam_M$ assigns to every Riemannian metric $g$ on $M$ a number $\Yam_M(g)\in\R\cup\set{-\infty}$, the \emph{Yamabe constant of $g$}, as follows. For each compactly supported not identically vanishing function $v\in C^\infty(M,\R_{\geq0})$, one defines
\[
\Yamfunc_g(v) \define \frac{1}{\norm{v}_{L^p(g)}^2}\int_M\Big(a_n\abs{\diff v}_g^2 +\scal_g\,v^2\Big)\mathop{\diff\mu_g} \in\R ,
\]
where $p=p_n\define\frac{2n}{n-2}$ and $a_n\define\frac{4(n-1)}{n-2}$ and $\scal_g$ denotes the scalar curvature of $g$. The Yamabe constant of $g$ is
\[
\Yam_M(g) \define \inf\Bigset{\Yamfunc_g(v) \Bigsuchthat v\in C^\infty_c(M,\R_{\geq0})\without\set{0}} \in\R\cup\set{-\infty} .
\]
$\Yam_M(g)$ depends only on the conformal class of $g$. The \emph{$\sigma$-invariant of $M$} is
\[
\sigma(M)\define\sup\bigset{\Yam_M(g) \bigsuchthat g\in\Metr(M)} ,
\]
where $\Metr(M)$ denotes the set of Riemannian metrics on $M$. Every metric $g$ on an $n$-manifold satisfies
\[
\Yam_M(g) \leq \sigma(S^n) = \Yam_{S^n}(g_\mt{st}) = n(n-1)\vol(S^n,g_\mt{st})^{2/n} ,
\]
where $g_\mt{st}$ is the standard metric on the $n$-sphere $S^n$. (See Section \ref{preliminaries} for details and references.)

\smallskip
In the case when $M$ is compact without boundary, the Yamabe constant and $\sigma$-invariant have been studied in hundreds of articles; cf.\ e.g.\ \cite{AkutagawaIshidaLeBrun, AmmannDahlHumbert, BrayNeves, BrendleMarques, PeteanRuiz} and the reference lists therein. Several of these works involve also Yamabe constants of noncompact manifolds as a tool. Some articles where the noncompact case has been investigated for its own sake are \cite{Akutagawa, AkutagawaBotvinnik2003, Grosse09, Grosse11, Kim1996, Kim1997, Kim2000}. In most cases the focus was on special classes of noncompact manifolds and/or metrics, e.g. $\R\times N$ with compact $N$, coverings of closed manifolds, or manifolds of bounded geometry. The aim of the present article is to state and prove several facts which hold for all manifolds and metrics.

\smallskip
One of these results is that the functional $\Yam_M$ is continuous in a suitable sense. In the case of compact $M$, this was proved by {\BerardBergery} \cite[Proposition 7.2]{BerardBergery}. He stated only continuity with respect to the $C^\infty$-topology on the space of metrics, but the proof works obviously even for the (coarser) $C^2$-topology; in this form the result is also given in \cite[Proposition 4.31]{Besse}. The proof is not completely trivial, because of the infimum that occurs in the definition of $\Yam_M$. But it is still reasonably straightforward, and the application of Moser's lemma suggested in both references is not really necessary.

\smallskip
In the present article, we discuss the continuity of $\Yam_M$ on noncompact manifolds $M$, where one has to distinguish between the usual (metrizable) \emph{compact-open} $C^2$-topology and the \emph{fine} (also known as \emph{strong} or \emph{Whitney}) $C^2$-topology, which is neither metrizable nor connected; cf.\ Section \ref{topologies} for a review. One can also consider another natural topology on $\Metr(M)$, which we call the \emph{uniform $C^k$-topology}; see Section \ref{topologies}. If $M$ is noncompact, this topology is strictly finer than the compact-open $C^k$-topology and strictly coarser than the fine $C^k$-topology.

\medskip
A straightforward generalization of \BerardBergery's arguments yields the following result:
\begin{theorem} \label{cocontinuity}
Let $M$ be a nonempty manifold of dimension $\geq3$. Then $\Yam_M$ is upper semicontinuous with respect to the compact-open $C^2$-topology on $\Metr(M)$. If $M$ is compact or $\Yam_M(g)=-\infty$, then $\Yam_M$ is continuous at $g$ with respect to the compact-open $C^2$-topology. If $M$ is noncompact and $\Yam_M(g)>-\infty$, then $\Yam_M$ is not continuous at $g$ for any compact-open $C^k$-topology on $\Metr(M)$ with $k\in\N\cup\set{\infty}$.
\end{theorem}

The Yamabe map has better continuity properties with respect to the \emph{uniform} $C^2$-topology (recall that for $f\in C^\infty(M,\R)$, the function $f_-\in C^0(M,\R_{\geq0})$ is defined by $f_-(x)=-\min\set{0,f(x)}$):
\begin{theorem} \label{uniformcontinuity}
Let $M$ be a nonempty manifold of dimension $n\geq3$. Then the Yamabe map $\Yam_M$ is upper semicontinuous with respect to the uniform $C^2$-topology on $\Metr(M)$. At every metric $g\in\Metr(M)$ which satisfies $\Yam_M(g)=-\infty$ or admits constants $\eps,c\in\R_{>0}$ with $\abs{\Ric_g}_g \leq c(1+\abs{\scal_g})$ and $\norm{(\scal_g-\eps)_-}_{L^{n/2}(g)} < \infty$, the Yamabe map is continuous with respect to the uniform $C^2$-topology.
\end{theorem}

However, there exist metrics at which the Yamabe map is not continuous for any uniform $C^k$-topology with $k\in\N\cup\set{\infty}$. Such metrics can have scalar curvature $0$ and bounded Ricci curvature, so the sufficient criterion above cannot be generalized to $\eps=0$:
\begin{exampleit} \label{uniformexample}
Let $n\geq4$, let $N$ be a nonempty closed $(n-1)$-manifold with $\sigma(N)>0$. Then $N$ admits a Riemannian metric $h$ with $\scal_h=0$ such that for the product metric $g\define h+\diff t^2$ on $M\define N\times\R$, the Yamabe map $\Yam_M$ is not continuous at $g$ for any uniform $C^k$-topology on $\Metr(M)$ with $k\in\N\cup\set{\infty}$.
\end{exampleit}

Even with respect to the \emph{fine} $C^2$-topology, it is not obvious that the Yamabe map is continuous at every metric: the infimum in the definition makes the situation on noncompact manifolds even more nonlocal than in the compact case. An argument sharper than \BerardBergery's yields our main result:
\begin{theorem} \label{finecontinuity}
Let $M$ be a nonempty manifold of dimension $\geq3$. Then the Yamabe map $\Yam_M$ is continuous with respect to the fine $C^2$-topology on the space of Riemannian metrics on $M$.
\end{theorem}
This shows that the fine $C^2$-topology is the correct topology in the context of the Yamabe map on noncompact manifolds, as one might have expected. Therefore we do not mention other topologies on $\Metr(M)$ in the following results.

\smallskip
Theorem \ref{finecontinuity} implies that for each $r\in\R\cup\set{-\infty}$ the set of $g\in\Metr(M)$ with $\Yam_M(g)=r$ is closed with respect to the fine $C^2$-topology on $\Metr(M)$. For $r=-\infty$, a stronger statement is true:

\begin{theorem} \label{imagetheoremclopen}
Let $M$ be a nonempty manifold of dimension $\geq3$. Then the set of $g\in\Metr(M)$ with $\Yam_M(g)=-\infty$ is open and closed with respect to the fine $C^2$-topology on $\Metr(M)$.
\end{theorem}

\medskip
S.\ Kim \cite{Kim1996,Kim1997} introduced another, closely related, functional $\Yaminfty=\Yaminfty_M$ on the space $\Metr(M)$ of Rie\-mann\-ian metrics on a noncompact $n$-manifold $M$: For a chosen compact exhaustion $(K_i)_{i\in\N}$ of $M$, one defines
\[
\Yaminfty_M(g) \define \lim_{i\to\infty}\Yam_{M\without K_i}(g) \;\in \cci{-\infty}{\,\sigma(S^n)} ,
\]
where the restriction of $g$ to $M\without K_i$ is suppressed in the notation. The limit exists and does not depend on the chosen exhaustion (cf. \ref{Yaminftywelldef} below). We call $\Yaminfty_M(g)$ the \emph{Yamabe constant at infinity} of $g$.

\begin{theorem} \label{locallyconstant}
Let $M$ be a noncompact manifold of dimension $\geq3$. Then $\Yaminfty_M$ is locally constant (in particular continuous) with respect to the fine $C^2$-topology on $\Metr(M)$.
\end{theorem}
In contrast, $\Yam_M$ is certainly not locally constant, because $\Yam_M(g)$ can be changed continuously by modifying $g$ on any compact subset $K$ of $M$ while keeping it fixed outside $K$.

\smallskip
Several general statements hold for the Yamabe constant and the Yamabe constant at infinity:
\begin{theorem} \label{gtheorem}
Every Riemannian metric $g$ on a noncompact manifold of dimension $n\geq3$ satisfies:
\begin{enumerate}
\item\label{gtheoremineq} $-\norm{(\scal_g)_-}_{L^{n/2}(g)} \leq \Yam_M(g) \leq \Yaminfty_M(g)$.
\item\label{gtheoreminftya} If $\Yaminfty_M(g)<0$, then $\Yaminfty_M(g)=-\infty$.
\item\label{gtheoreminftyb} If $\Yam_M(g)=-\infty$, then $\Yaminfty_M(g)=-\infty$.
\end{enumerate}
\end{theorem}
For some remarks and a conjecture related to Theorem \ref{gtheorem}(\ref{gtheoremineq}), see Section \ref{gtheoremproof}.

\begin{theorem} \label{imagetheorem}
Let $M$ be a nonempty manifold of dimension $n\geq3$ each of whose connected components is noncompact. Then:
\begin{enumerate}
\item\label{imagetheoremsigma} The image of $\Yam_M$ is an interval which contains $-\infty$ and $0$. Thus $0\leq\sigma(M)\leq\sigma(S^n)$.
\item\label{imagetheorempositive} If $M$ is diffeomorphic to an open subset of a compact $n$-manifold, then $0<\sigma(M)$.
\end{enumerate}
\end{theorem}

\begin{remarks}{\ } \label{introremarks}
\begin{enumerate}
\item\label{introremarkone} If a metric $g$ on a (possibly noncompact) manifold $M$ of dimension $\geq6$ satisfies $\Yam_M(g)=\sigma(S^n)$, then $g$ is locally conformally flat, by Aubin's local argument \cite{Aubin}, \cite[proof of Thm.~B]{LeeParker}. Whether this generalizes to dimension $3$, $4$, or $5$ is unclear. A simply connected $n$-manifold $M$ with $n\geq3$ admits a locally conformally flat metric if and only if it can be immersed into $S^n$ \cite[pp.~49--50]{SchoenYau1988}. A noncompact connected $n$-manifold can be immersed into $S^n$ if and only if it is parallelizable; cf.\ \ref{parallelizable} below. Thus for many noncompact manifolds $M$ of dimension $n\geq6$ (e.g.\ all simply connected nonparallelizable ones), $\sigma(S^n)$ does not lie in the image of $\Yam_M$.
\item We suspect that $\sigma(M)=\sigma(S^n)$ holds for every noncompact connected $n$-manifold $M$; then for such $M$, the image of $\Yam_M$ would always be either $\coi{-\infty}{\sigma(S^n)}$ or $\cci{-\infty}{\sigma(S^n)}$.
\item By Theorems \ref{imagetheorem}(\ref{imagetheoremsigma}) and \ref{gtheorem}, the image of $\Yaminfty_M$ contains $-\infty$ and a nonnegative number, but no negative real number. Hence it is not an interval. We don't know any other lower or upper bound on the number of ``gaps'' it has. Nor do we know whether there exists a manifold $M$ for which the image of $\Yaminfty_M$ contains an interval of nonzero length. We suspect that every noncompact connected $n$-manifold $M$ admits a Riemannian metric $g$ with $\Yaminfty_M(g)=\sigma(S^n)$. For each such $M$ which is diffeomorphic to an open subset of a compact manifold, this is true: \cite[Theorem 3.1]{Kim2000} implies that $\Yaminfty_M(g)=\sigma(S^n)$ holds for any $g$ which is the pullback of a metric on the compact manifold (note that the completeness assumption in that theorem is irrelevant because each conformal class contains a complete metric).
\end{enumerate}
\end{remarks}

In the following Sections \ref{preliminaries}, \ref{topologies}, we review relevant definitions and basic facts, in particular about the Yamabe constant and topologies on $\Metr(M)$. The rest of the article contains the proofs of the theorems and of Example \ref{uniformexample}. The proofs are not presented in the order of the theorem numbers but in such a way that every result has been proved before it is applied in other proofs.

\section{Preliminaries} \label{preliminaries}

\begin{conventions}
$0\in\N$. The words \emph{manifold}, \emph{metric}, \emph{map}, \emph{section} etc.\ mean smooth objects, except when explicitly stated otherwise. Manifolds are pure-dimensional and second countable and do not have a boundary; thus the notions \emph{closed manifold} and \emph{compact manifold} are synonymous.
\end{conventions}

\begin{exhaustions}
Let $M$ be an $n$-manifold. A \emph{compact exhaustion} of $M$ is a sequence $(K_i)_{i\in\N}$ of compact subsets $K_i$ of $M$ such that for every $i\in\N$, \,$K_i$ is contained in the interior of $K_{i+1}$ in $M$, and such that $M=\bigcup_{i\in\N}K_i$.

\smallskip
Every manifold admits a compact exhaustion. Every compact exhaustion $(K_i)_{i\in\N}$ of a \emph{compact} manifold $M$ satisfies $K_i=M$ for all sufficiently large $i$. If a compact exhaustion $(K_i)_{i\in\N}$ of a connected manifold $M$ satisfies $K_{i+1}=K_i\neq\leer$ for some $i$, then $M=K_i$ (because $K_i$ is open, closed and nonempty), thus $M$ is compact.
\end{exhaustions}

\begin{semicont}
Let $X$ be a topological space, let $x\in X$. A function $f\colon X\to\R\cup\set{-\infty}$ is \emph{upper} [resp.\ \emph{lower}] \emph{semicontinuous at $x$} iff the following is true:
\begin{itemize}
\item If $f(x)\in\R$, then for every $\eps\in\R_{>0}$ there exists a neighborhood $U$ of $x$ such that $f(y)\leq f(x)+\eps$ [resp. $f(y)\geq f(x)-\eps$] holds for all $y\in U$.
\item If $f(x)=-\infty$, then for every $c\in\R$ there exists a neighborhood $U$ of $x$ such that $f(y)\leq c$ [resp. $f(y)\geq -\infty$] holds for all $y\in U$.
\end{itemize}
$f$ is \emph{upper} [resp.\ \emph{lower}] \emph{semicontinuous} iff it is upper [resp.\ lower] semicontinuous at each $x\in X$.

\smallskip
In the article \cite{BerardBergery}, the notions of upper and lower semicontinuity are mixed up. This has been corrected in \cite[Proposition 4.31]{Besse}.
\end{semicont}

In addition to the notations which occurred in the introduction, we will use the following ones:
\begin{notations}
Let $M$ be an $n$-manifold.
\begin{itemize}
\item Our sign convention for the Laplacian $\laplace_g\colon C^\infty(M,\R) \to C^\infty(M,\R)$ with respect to a Riemannian metric $g$ is $\laplace_gu = -\divergence_g(\diff u)$, i.e.\ $\laplace_gu = -\sum_{i=1}^n\frac{\partial^2u}{\partial x_i^2}$ in Euclidean space.
\item $\abs{\Ric_g}_g\in C^0(M,\R_{\geq0})$ is defined by $\abs{\Ric_g}_g(x) = \big(\sum_{i,j=1}^n\Ric_g(e_i,e_j)^2\big)^{1/2}$, where $\Ric_g$ is the Ricci tensor of $g$ and $(e_1,\dots,e_n)$ is any $g$-orthonormal basis of $T_xM$.
\item Let $q\in\R_{\geq1}$. The $L^q(g)$-norm of $v\in C^0(M,\R)$ is $\norm{v}_{L^q(g)} \define \big(\int_Mv^q\,\diff\mu_g\big)^{1/q} \in [0,\infty]$, where $\diff\mu_g$ denotes the density on $M$ induced by $g$. The $L^q(g)$-norm of a $1$-form $\alpha$ on $M$ is $\norm{\alpha}_{L^q(g)} \define \norm{\abs{\alpha}_g}_{L^q(g)}$. For a measurable subset $A$ of $M$, the norm $\norm{.}_{L^q(A;g)}$ of a function or $1$-form on $M$ is defined in the same way as $\norm{.}_{L^q(g)}$, just with $\int_A$ instead of $\int_M$.

    \smallskip
\item
For Riemannian metrics $g,h$ on $M$, \,$\F{h}{g}\in C^\infty(M,\R_{>0})$ is defined by $\diff\mu_h = \F{h}{g}\,\diff\mu_g$.
\item For $f\in C^0(M,\R)$, the functions $f_\pm\in C^0(M,\R_{\geq0})$ are defined by $f_+(x)=\max\set{0,f(x)}$ and $f_-(x)=-\min\set{0,f(x)}$, respectively.

    \smallskip
\item Let $k\in\N$. We define the $C^k(g)$-norm of a (smooth) section $h$ in the vector bundle $\Sym^2T^\ast M$ over $M$ by $\norm{h}_{C^k(g)}\define \sum_{i=0}^k\sup\bigset{\abs{\nabla^ih}_g(x) \bigsuchthat x\in M} \in [0,\infty]$, where $\nabla^ih = \nabla\cdots\nabla h$ denotes the $i$th covariant derivative of $h$ with respect to the Levi-Civita connection of $g$.

    \smallskip\noindent
    For $K\subseteq M$, the ``norm'' $\norm{h}_{C^k(K;g)}$ of a section $h$ in $\Sym^2T^\ast M\to M$ is defined in the same way as $\norm{h}_{C^k(g)}$, just with the suprema over $M$ replaced by suprema over $K$. If $K$ is compact, then all values of $\norm{.}_{C^k(K;g)}$ are finite and $\norm{.}_{C^k(g)}$ is indeed a norm, and all such norms induced by different metrics $g$ are equivalent.
\end{itemize}
\end{notations}

\begin{yamabe}
Notation and terminology are not standardized: the letters $\mu$ and $Q$ are often used instead of our $\Yam$, definitions might differ by a factor $a_n$, and some people call $\Yam_M(g)$ the \emph{Yamabe invariant}, whereas others call the $\sigma$-invariant the \emph{Yamabe invariant of $M$}. We therefore avoid the term \emph{Yamabe invariant} entirely. The \emph{Yamabe constant}, \emph{$\sigma$-invariant} terminology and the letter $Y$ seem to become more and more standard anyway.

\smallskip
Let $M$ be a nonempty $n$-manifold. The Yamabe constant is a conformal invariant: For every $g\in\Metr(M)$ and $u\in C^\infty(M,\R_{>0})$, the conformal metric $\tilde{g}\define u^{4/(n-2)}g$ satisfies $\Yamfunc_{\tilde{g}}(v) = \Yamfunc_g(uv)$ for all $v\in C^\infty_c(M,\R_{\geq0})\without\set{0}$, hence $\Yam_M(g) = \Yam_M(\tilde{g})$. (This follows by partial integration from $\diff\mu_{\tilde{g}} = u^{2n/(n-2)}\diff\mu_g$ and $\scal_{\tilde{g}} = u^{-(n+2)/(n-2)}(a_n\laplace_gu +\scal_gu)$ and $\abs{\diff w}_{\tilde{g}}^2 = u^{-4/(n-2)}\abs{\diff w}_g^2$.)

\smallskip
Hence also the Yamabe constant at infinity of a noncompact manifold is a conformal invariant.

\smallskip
$\Yamfunc_g(v) = \Yamfunc_g(cv)$ holds for all $g\in\Metr(M)$ and $c\in\R_{>0}$ and $v\in C^\infty(M,\R_{\geq0})\without\set{0}$. This implies
\[
\Yam_M(g) = \inf\bigset{\Yamfunc_g(v) \bigsuchthat v\in C^\infty_c(M,\R_{\geq0}), \,\norm{v}_{L^{2n/(n-2)}(h)}=1}
\]
for any metric $h\in\Metr(M)$. We will use this fact repeatedly in the present article.

\smallskip
Whenever $g$ is a metric on $M$ and $U$ is a nonempty open subset of $M$, we will denote the Yamabe constant of the restriction of $g$ to $U$ by $\Yam_U(g)$; i.e., we suppress the restriction of the metric in our notation. The same convention applies to $\Yaminfty$.
\end{yamabe}

\begin{factit} \label{increasing}
Let $M,N$ be nonempty $n$-manifolds with $n\geq3$, let $\iota\colon N\to M$ be a smooth embedding. Then each Riemannian metric $g$ on $M$ satisfies $\Yam_N(\iota^\ast g)\geq\Yam_M(g)$. Thus $\sigma(N)\geq\sigma(M)$.
\end{factit}
\begin{proof}
For every $v\in C^\infty_c(N,\R_{\geq0})\without\set{0}$, we consider the function $\hat{v}\in C^\infty_c(M,\R_{\geq0})\without\set{0}$ defined by $\hat{v}\compose\iota = v$ and $\supp(\hat{v}) = \iota(\supp(v))$. Since $\Yamfunc_g(\hat{v}) = \Yamfunc_{\iota^\ast g}(v)$, we obtain $\Yam_N(\iota^\ast g)\geq\Yam_M(g)$.
\end{proof}

\begin{void}
As mentioned in the introduction, $\Yam_M(g)\leq\sigma(S^n)$ holds for every nonempty $n$-manifold $M$ and $g\in\Metr(M)$. This is stated and proved for closed $M$ in \cite[Lemma 3.4]{LeeParker}, and the proof for arbitrary $M$ consists of exactly the same local argument involving test functions with supports in a small ball.
\end{void}

\begin{void} \label{Yaminftywelldef}
Let $M$ be a noncompact $n$-manifold, let $(K_i)_{i\in\N}$ be a compact exhaustion of $M$, let $g\in\Metr(M)$. In the definition of the Yamabe constant at infinity $\Yaminfty_M(g)$, the sequence $\big(\Yam_{M\without K_i}(g)\big)_{i\in\N}$ in $\R\cup\set{-\infty}$ is monotonically increasing by Fact \ref{increasing}, because $M\without K_{i+1}\subseteq M\without K_i$ holds for each $i\in\N$. Since the sequence is also bounded from above by $\sigma(S^n)$, the limit $\lim_{i\to\infty}\Yam_{M\without K_i}(g)$ exists in $[-\infty,\sigma(S^n)]$.

\smallskip
Let $(K'_i)_{i\in\N}$ be another compact exhaustion of $M$. For every $i\in\N$, there exists a number $j(i)\in\N$ with $K'_i\subseteq K_{j(i)}$. Fact \ref{increasing} yields $\Yam_{M\without K'_i}(g) \leq \Yam_{M\without K_{j(i)}}(g) \leq \lim_{j\to\infty}\Yam_{M\without K_j}(g)$ for each $i$, hence $\lim_{i\to\infty}\Yam_{M\without K'_i}(g) \leq \lim_{i\to\infty}\Yam_{M\without K_i}(g)$. For symmetry reasons the reversed inequality holds as well. Thus $\Yaminfty_M(g)$ does not depend on the chosen exhaustion, as we claimed in the introduction.
\end{void}

\begin{remark}
Recall that we did not define $\Yam_M$ in the case when $M$ is empty; thus $\Yaminfty_M(g)$ is defined only for noncompact manifolds (because every compact exhaustion of a compact manifold $M$ is eventually constant $M$). For a fixed dimension $n$, a natural choice in the case $M=\leer$ would be $\Yam_\leer(g)\define\sigma(S^n)$ for the unique $g\in\Metr(\leer)$. Then the assumption of $M$ being nonempty could be omitted in the Theorems \ref{cocontinuity}, \ref{uniformcontinuity} and \ref{finecontinuity}. Moreover, $\Yaminfty_M(g)$ would be defined in the same way as above for each metric $g$ on a closed $n$-manifold $M$, and it would be equal to $\sigma(S^n)$.
\end{remark}

\begin{remark}
Without further comment we will often use Hölder's inequality in the following form: For $n\in\N_{\geq3}$, let $p=\frac{2n}{n-2}$. Then
\begin{align*}
\norm{v^2w}_{L^1(g)} &\leq \norm{v}_{L^p(g)}^2\norm{w}_{L^{n/2}(g)}
\end{align*}
hold for all manifolds $M$ and $g\in\Metr(M)$ and $v,w\in C^0(M,\R)$, because $1 = \frac{2}{p}+\frac{2}{n}$.
\end{remark}

In Remark \ref{introremarks}(\ref{introremarkone}), we made the following claim:
\begin{factit} \label{parallelizable}
Let $n\geq0$. A noncompact connected $n$-manifold can be immersed into $S^n$ if and only if it is parallelizable.
\end{factit}
\begin{proof}
Let $M$ be a noncompact connected $n$-manifold. First we prove that $M$ can be immersed into $S^n$ if and only if it can be immersed into $\R^n$. The ``if'' part is obvious. For ``only if'', let $f\colon M\to S^n$ be an immersion, let $x\in S^n$. The set $D\define f^{-1}(\set{x})$ is discrete and closed in $M$ because $f$ is a local diffeomorphism. Since $M$ is noncompact and connected, there exists an open subset $M'$ of $M\without D$ which is diffeomorphic to $M$ (choose a smooth triangulation of $M$, use a diffeomorphism $M\to M$ to move all elements of $D$ away from the $(n-1)$-skeleton, and apply \cite[Theorem 3.7]{Hirsch1961}). The map $f\restrict_{M'}\colon M\cong M'\to S^n\without\set{x}\cong\R^n$ is an immersion.

\smallskip
It remains to prove that $M$ can be immersed into $\R^n$ if and only if it is parallelizable. The ``only if'' part is true because the immersion pullback of a tangent frame on $\R^n$ is a tangent frame on $M$. The ``if'' part is an application of Smale--Hirsch immersion theory; cf.\ \cite[Theorem 4.7]{Hirsch1961}.
\end{proof}

\section{The three topologies} \label{topologies}

In this section we briefly review the compact-open and fine $C^k$-topologies. (The latter is also known as the \emph{strong} or \emph{Whitney} $C^k$-topology \cite{Hirsch}; we follow Gromov \cite{Gromov} in calling it the \emph{fine} $C^k$-topology.) After that, we define another natural topology on the set of Riemannian metrics, which we call the \emph{uniform $C^k$-topology}. It has probably been considered in the literature before, but we don't know where.

\begin{definition}[the fine $C^k$-topology]
Let $E$ be a fiber bundle over a manifold $M$, let $k\in\N\cup\set{\infty}$. The \emph{fine $C^k$-topology} on the set of (smooth) sections in $E$ is defined by declaring at each section $s$ a neighborhood basis $\mathscr{B}_k(s)$ as follows \cite[p.~9]{Spring}. A section $\xi$ in the $k$-jet bundle $J^kE$ over $M$ can be identified with its graph, i.e.\ with the image $\graph(\xi)$ of $\xi$ in the total space of $J^kE$. We define $\mathcal{U}_k(s)$ to be the set of open neighborhoods of $\graph(j^ks)$ in the total space of $J^kE$. For $U\in\mathcal{U}_k(s)$, we consider the set $\mathcal{N}_U$ of sections $\tilde{s}$ in $E$ with $\graph(j^k\tilde{s})\subseteq U$. Then
\[
\mathscr{B}_k(s) \define \bigset{\mathcal{N}_U \bigsuchthat U\in\mathcal{U}_k(s)} .
\]
$\Metr(M)$ is the set of sections in the fiber bundle $\Sym^2_+T^\ast M$ over $M$, whose fiber over $x$ consists of the positive definite symmetric bilinear forms on $T_x^\ast M$. Thus a fine $C^k$-topology is defined on $\Metr(M)$.
\end{definition}

\begin{examples} \label{fineexamples}
Let $E$ be a fiber bundle over a manifold $M$, let $k\in\N\cup\set{\infty}$, let $F\in C^0(J^kE,\R_{\geq0})$, let $\eps\in C^0(M,\R_{>0})$, let $s$ be a section in $E$ with $F\compose j^ks = 0$. Then the set of sections $\tilde{s}$ in $E$ with $F\compose j^k\tilde{s} < \eps$ is an open neighborhood of $s$ with respect to the fine $C^k$-topology: since $F$, $\eps$ and the projection $\pr\colon J^kE\to M$ are continuous, the set $U = \set{\eta\in J^kE \suchthat F(\eta)<\eps(\pr(\eta))}$ is an open neighborhood of $\graph(j^ks)$, and thus the set $\mathcal{N}_U$ of sections $\tilde{s}$ in $E$ with $F\compose j^k\tilde{s} < \eps$ is fine $C^k$-open.

\smallskip
For instance, let $g\in\Metr(M)$. If $F\colon J^2\Sym^2_+T^\ast M\to\R$ is one of the following maps, then the set of $h\in\Metr(M)$ with $F\compose j^2h < \eps$ is an open neighborhood of $g$ with respect to the fine $C^2$-topology:
\begin{enumerate}
\item $F\colon j^2_xh \mapsto \abs{\scal_h(x)-\scal_g(x)}$.
\item $F\colon j^2_xh \mapsto \max\bigset{\bigabs{\abs{\alpha}_h^2-1} \bigsuchthat \alpha\in T^\ast_xM, \,\abs{\alpha}_g=1}$.
\item $F\colon j^2_xh \mapsto \bigabs{\F{h}{g}(x)-1}$.
\item $F\colon j^2_xh \mapsto \bigabs{\diff\big(\F{h}{g}\big)}_g(x)$.
\item $F\colon j^2_xh \mapsto \bigabs{\laplace_g\big(\F{h}{g}\big)(x)}$.
\end{enumerate}
(The maps (2), (3) even define fine $C^0$-neighborhoods, and (4) defines a fine $C^1$-neighborhood. But we will later use only that they are fine $C^2$-neighborhoods.) All these maps $F$ are well-defined because the right-hand sides contain at most second derivatives of $h$, and the continuity is easy to check in each case.
\end{examples}

\begin{definition}[the compact-open $C^k$-topology]
For topological spaces $X,Y$, the compact-open topology on the set of continuous maps $X\to Y$ is well-known. Let $E$ be a fiber bundle over a manifold $M$, let $k\in\N\cup\set{\infty}$. We consider the map $j^k$ from the set of (smooth) sections in $E$ to sections in $J^kE$ which sends each $s$ to its $k$-jet prolongation $j^ks$, and we equip the set of sections in $J^kE$ with the subspace topology of the compact-open topology on the space of continuous maps $M\to J^kE$. The \emph{compact-open $C^k$-topology} on the set of sections in $E$ is the coarsest topology which makes $j^k$ continuous.
\end{definition}

The following basic facts are well-known \cite[p.~35--36]{Hirsch}:
\begin{facts} \label{topologyfacts}
Let $k\in\N\cup\set{\infty}$. The compact-open $C^k$-topology on $\Metr(M)$ is metrizable and path-connected (for $g_0,g_1\in\Metr(M)$, the path $(g_t)_{t\in[0,1]}$ given by $g_t\define (1-t)g_0 +tg_1$ is continuous). For $k<\infty$, a sequence $(g_i)_{i\in\N}$ in $\Metr(M)$ converges to $g\in\Metr(M)$ with respect to the compact-open $C^k$-topology if and only if for some (and hence every) auxiliary metric $h\in\Metr(M)$ (e.g.\ $h=g$) and for every compact subset $K$ of $M$, the sequence $(\norm{g_i-g}_{C^k(K;h)})_{i\in\N}$ converges to $0$. If $M$ is compact, then the fine $C^k$-topology on $\Metr(M)$ is equal to the compact-open $C^k$-topology on $\Metr(M)$. If $M$ is noncompact, then the fine $C^k$-topology on $\Metr(M)$ is (much) finer than the compact-open $C^k$-topology. For instance it is neither first countable (hence not metrizable) nor connected. For metrics $g_0,g_1\in\Metr(M)$ which differ outside each compact subset of $M$, every path from $g_0$ to $g_1$, in particular the map $[0,1]\to\Metr(M)$ given by $t\mapsto(1-t)g_0+tg_1$, is not fine $C^k$-continuous. The compact-open (resp.\ fine) $C^\infty$-topology (considered as a set of open sets) on $\Metr(M)$ is the union of all compact-open (resp.\ fine) $C^k$-topologies on $\Metr(M)$ with $k\in\N$. For $l\in\N\cup\set{\infty}$ with $l\geq k$, the compact-open $C^l$-topology on $\Metr(M)$ is finer than the compact-open $C^k$-topology, and the fine $C^l$-topology on $\Metr(M)$ is finer than the fine $C^k$-topology.
\end{facts}

Consider a section $s_0$ in a fiber bundle $E$ over a noncompact manifold $M$. Each neighborhood of $s_0$ with respect to the compact-open $C^0$-topology contains sections $s$ such that the values $s(x)$ and $s_0(x)$ are, intuitively speaking, farther and farther away as $x$ tends to infinity in $M$. Whereas, again intuitively speaking, for each element $s$ of a typical neighborhood of $s_0$ with respect to the \emph{fine} $C^0$-topology, the values $s(x)$ and $s_0(x)$ become closer and closer as $x$ tends to infinity in $M$. (Similar intuitive statements involving derivatives of $s_0,s$ apply to the higher $C^k$-topologies.) A topology with the property that, for a typical element $s$ of a typical neighborhood of $s_0$, the distance of $s_0(x),s(x)$ stays uniform as $x$ tends to infinity can in general make sense only after one has equipped the fibers of $E$ with an auxiliary metric which defines what is meant by ``distance'' and ``uniform''. The resulting topology will then depend strongly on that auxiliary metric. But in the special situation where $E=\Sym^2_+T^\ast M$, a uniform topology can be defined without reference to an auxiliary metric:

\begin{definition}[the uniform $C^k$-topology]
Let $M$ be a manifold, let $k\in\N$. We define the \emph{uniform $C^k$-topology} on $\Metr(M)$ by declaring at each $g\in\Metr(M)$ a neighborhood basis $\mathscr{B}_k'(g)$: for $\eps\in\R_{>0}$, we let $\mathcal{N}_{g,\eps,k}\define \bigset{h\in\Metr(M) \bigsuchthat \norm{h-g}_{C^k(g)}<\eps}$ and $\mathscr{B}_k'(g)\define \set{\mathcal{N}_{g,\eps,k} \suchthat \eps\in\R_{>0}}$. We define the \emph{uniform $C^\infty$-topology} on $\Metr(M)$ to be the union of all uniform $C^k$-topologies (considered as sets of open sets) on $\Metr(M)$ with $k\in\N$.
\end{definition}
\begin{proof}[Proof that this defines a neighborhood basis of a topology on $\Metr(M)$]
Each $\mathscr{B}_k'(g)$ is nonempty, and each $\mathcal{N}_{g,\eps,k}$ contains $g$. For every two elements $\mathcal{N}_{g,\eps_0,k},\mathcal{N}_{g,\eps_1,k}$ of $\mathscr{B}_k'(g)$, the set $\mathcal{N}_{g,\eps_0,k}\cap\mathcal{N}_{g,\eps_1,k}$ contains an element of $\mathscr{B}_k'(g)$, namely $\mathcal{N}_{g,\min(\eps_0,\eps_1),k}$.
\end{proof}

The uniform $C^k$-topologies are natural objects in particular when one considers Riemannian metrics on product manifolds $M\times N$ with compact $M$ and noncompact $N$. The compact-open topologies are much too coarse to control the Yamabe constant even near product metrics, as Theorem \ref{cocontinuity} shows. Whereas the fine topologies are much too fine for instance for a reasonable discussion of $1$-parameter families of product metrics $g_M(t)\oplus g_N$ on $M\times N$, because they make such a $1$-parameter family continuous only if it is constant. In contrast, the uniform $C^k$-topology makes such a $1$-parameter family continuous if and only if $(g_M(t))_{t\in\R}$ is a $C^k$-continuous family (the fine/uniform/compact-open distinction plays no role here because $M$ is compact); moreover, it makes the Yamabe map continuous at many product metrics (provided $k\geq2$), as one can see from Theorem \ref{uniformexample}. This is what one would intuitively expect from a nice topology on $\Metr(M\times N)$. Unfortunately, Example \ref{uniformexample} shows that the uniform topologies do not make $\Yam_{M\times N}$ continuous at \emph{every} product metric.

\begin{factsit} \label{uniformfacts}
Let $M$ be a manifold, let $k,l\in\N\cup\set{\infty}$ with $l\geq k$. The uniform $C^k$-topology on $\Metr(M)$ is coarser than the uniform $C^l$-topology. It is finer than the compact-open $C^k$ topology, and it is coarser than the fine $C^k$-topology; in particular, it is equal to both these topologies if $M$ is compact. If $M$ is noncompact, then the uniform $C^k$-topology on $\Metr(M)$ is neither equal to any compact-open $C^r$-topology nor equal to any fine $C^r$-topology.
\end{factsit}
\begin{proof}
If $l\in\N$, then every uniform $C^k$-neighborhood $\mathcal{N}_{g,\eps,k}$ of $g\in\Metr(M)$ contains a uniform $C^l$-neighborhood of $g$, namely $\mathcal{N}_{g,\eps,l}$. Thus the uniform $C^k$-topology is coarser than the uniform $C^l$-topology if $l\in\N$. The same holds by definition of the uniform $C^\infty$-topology also for $l=\infty$.

\smallskip
Every uniform $C^k$-neighborhood $\mathcal{N}_{g,\eps,k}$ is a fine $C^k$-neighborhood of $g$: since $\abs{\nabla^i(h-g)}_g(x)$ depends continuously on $j^i_xh$, there exists a neighborhood $U$ of $\graph(j^kg)$ in $J^k\Sym^2_+T^\ast M$ such that the elements of
\[
\mathcal{N}_{g,\eps,k} = \Bigset{h\in\Metr(M) \Bigsuchthat \textstyle\sum_{i=0}^k\sup\bigset{ \abs{\nabla^i(h-g)}_g(x) \bigsuchthat x\in M} < \eps}
\]
are precisely those $h\in\Metr(M)$ with $\graph(j^kh)\subseteq U$. Thus the uniform $C^k$-topology is coarser than the fine $C^k$-topology.

\smallskip
For $K\subseteq M$ and $U\subseteq J^k\Sym^2_+T^\ast M$, let $\mathcal{M}_{K,U,k}\define \set{h\in\Metr(M) \suchthat \forall x\in K\colon j^k_xh\in U}$. By definition of the compact-open $C^k$-topology, the sets $\mathcal{M}_{K,U,k}$ such that $K\subseteq M$ is compact and $U\subseteq J^k\Sym^2_+T^\ast M$ is open form a subbase of the compact-open $C^k$-topology. We claim that each of these subbase elements is uniform $C^k$-open. In order to check this, we consider an element $g$ of $\mathcal{M}_{K,U,k}$. Since $U$ is open and $K$ is compact, there exists an $\eps\in\R_{>0}$ such that $\abs{\nabla^i(h-g)}_g(x)<\frac{\eps}{k+1}$ holds for all $h\in\mathcal{M}_{K,U,k}$ and $x\in K$ and $i\in\set{0,\dots,k}$; here $\nabla$ denotes the Levi-Civita connection of $g$. Therefore the uniform $C^k$-open set $\mathcal{N}_{g,\eps,k}$ is obviously contained in $\mathcal{M}_{K,U,k}$. As this is true for every $g\in\mathcal{M}_{K,U,k}$, the set $\mathcal{M}_{K,U,k}$ is indeed uniform $C^k$-open. This proves that the uniform $C^k$-topology is finer than the compact-open $C^k$-topology.

\smallskip
The uniform $C^k$-topology is not equal to any fine $C^r$-topology if $M$ is noncompact, because the uniform $C^k$-topology is by definition first countable, whereas the fine $C^r$-topology is not if $M$ is noncompact; cf.\ Facts \ref{topologyfacts}.

\smallskip
The uniform $C^k$-topology is not equal to any compact-open $C^r$-topology if $M$ is noncompact: We take any metric $g$ on $M$ and any $f\in C^\infty(M,\R_{>0})$ which is not bounded from above, and we consider $\gamma\colon[0,1]\to\Metr(M)$ given by $\gamma(t)\define (1-t)g +tfg$. This $\gamma$ is compact-open $C^r$-continuous at $0$, because $\lim_{t\to0}\norm{\gamma(t)-\gamma(0)}_{C^r(K;\gamma(0))} = \lim_{t\searrow0}t\norm{(f-1)g}_{C^r(K;g)} = 0$ holds for every compact subset $K$ of $M$. But $\gamma$ is not uniform $C^k$-continuous at $0$: For the neighborhood $\mathcal{N}_{g,1,k}\subseteq \mathcal{N}_{g,1,0}$ of $g=\gamma(0)$, there does not exist any $\delta\in\R_{>0}$ with $\forall t\in[0,\delta]\colon \gamma(t)\in\mathcal{N}_{g,1,k}$. That's because
\[
\bignorm{\gamma(t)-g}_{C^0(g)} = t\bignorm{(1-f)g}_{C^0(g)}
= t\sup_{x\in M} \bigabs{(1-f)g}_g(x)
= t\sup_{x\in M} \sqrt{\dim(M)}\,\abs{f(x)-1} = \infty
\]
for each $t\in\R_{>0}$. Thus the uniform $C^k$-topology differs indeed from the compact-open topologies.
\end{proof}

We leave it to the interested reader to state and prove further properties of the uniform $C^k$-topology. In the present article it serves only as an instructive intermediate step between the compact-open and fine topologies which clarifies nicely the continuity properties of the Yamabe map, in particular at product metrics on product manifolds one of whose factors is compact. All we have to know in that context are the facts listed above and Lemma \ref{uniformlemma} below.

\section{Proof of upper semicontinuity}

The proof of the following fact generalizes directly the one for closed manifolds \cite[Proposition 7.2]{BerardBergery}.

\begin{lemma} \label{coupper}
Let $M$ be a nonempty manifold of dimension $n\geq3$. Let $\Metr(M)$ be equipped with the compact-open $C^2$-topology. Then $\Yam_M$ is upper semicontinuous. In particular, $\Yam_M$ is continuous at each metric $g$ with $\Yam_M(g)=-\infty$.
\end{lemma}

\begin{proof}
For each $v\in C_c^\infty(M,\R_{\geq0})\without\set{0}$, the map $\Metr(M)\to\R$ given by $g\mapsto \Yamfunc_g(v)$ is continuous with respect to the compact-open $C^2$-topology: Since this topology is metrizable, it suffices to show that whenever a sequence $(g_i)_{i\in\N}$ in $\Metr(M)$ converges to $g$, then $\lim_{i\to\infty}\Yamfunc_{g_i}(v) = \Yamfunc_g(v)$. For the compact set $K\define\supp(v)$, the convergence of $(g_i)_{i\in\N}$ to $g$ implies $\lim_{i\to\infty}\norm{g_i-g}_{C^2(K;g)}=0$, which yields obviously $\lim_{i\to\infty}\norm{\scal_{g_i}-\scal_g}_{C^0(K)} = 0$ and $\lim_{i\to\infty}\bignorm{\abs{\diff v}_{g_i}^2 -\abs{\diff v}_g^2}_{C^0(K)} = 0$ and $\lim_{i\to\infty}\bignorm{\F{g_i}{g}-1}_{C^0(K)} = 0$, thus $\lim_{i\to\infty}\Yamfunc_{g_i}(v) = \Yamfunc_g(v)$. Hence $g\mapsto \Yamfunc_g(v)$ is indeed continuous.

\smallskip
Recall that whenever $X$ is a topological space and $Y$ is a nonempty set and $f\colon X\times Y\to\R$ has the property that $f(.,y)\colon X\to\R$ is continuous for every $y\in Y$, then the map $X\to\R\cup\set{-\infty}$ given by $x\mapsto \inf\set{f(x,y) \suchthat y\in Y}$ is upper semicontinuous \cite[\S IV.6.2, Corollary to Thm.~4]{BourbakiGT1}. Applying this to $X=\Metr(M)$ and $Y=C_c^\infty(M,\R_{\geq0})\without\set{0}$ and $f\colon (g,v)\mapsto E_g(v)$, we see that $\Yam_g$ is upper semicontinuous with respect to the compact-open $C^2$-topology.
\end{proof}

\begin{corollary} \label{upper}
Let $M$ be a nonempty manifold of dimension $\geq3$, let $k\in\N_{\geq2}\cup\set{\infty}$. Let $\Metr(M)$ be equipped either with the compact-open $C^k$-topology or with the uniform $C^k$-topology or with the fine $C^k$-topology. Then $\Yam_M$ is upper semicontinuous. It is continuous at each metric $g$ with $\Yam_M(g)=-\infty$.
\end{corollary}
\begin{proof}
Each of the considered topologies is finer than the compact-open $C^2$-topology.
\end{proof}

\section{Proof of Theorem \ref{gtheorem}} \label{gtheoremproof}

\begin{proof}[Proof of Theorem \ref{gtheorem}(\ref{gtheoremineq})]
By Fact \ref{increasing}, $\Yam_M(g)\leq \Yam_{M\without K}(g)$ holds for all compact subsets $K$ of $M$. Thus $\Yam_M(g) \leq \Yaminfty_M(g)$. In order to prove $-\norm{(\scal_g)_-}_{L^{n/2}(g)} \leq \Yam_M(g)$, we apply the H\"older inequality to each $v\in C_c^\infty(M,\R_{\geq0})\without\set{0}$ (using $\norm{v}_{L^p}^2 = \norm{v^2}_{L^{p/2}}$ and $\frac{2}{p} +\frac{2}{n} = 1$ for $p = \frac{2n}{n-2}$) and take the infimum over $v$ afterwards:
\[
E_g(v)
= \frac{\int_M\big(a_n\abs{\diff v}_g^2 +\scal_g\,v^2\big)\diff\mu_g}{\norm{v}_{L^p(g)}^2}
\geq  -\frac{\int_M (\scal_g)_-\,v^2\,\diff\mu_g}{\norm{v}_{L^p(g)}^2}
\geq -\bignorm{(\scal_g)_-}_{L^{n/2}(g)} . \qedhere
\]
\end{proof}

\begin{remarks}
In the estimate $\Yam_M(g)\leq \Yaminfty_M(g)$, equality is possible. Clearly we have $\Yam_M(g) = \Yaminfty_M(g)$ if $\Yam_M(g)=\sigma(S^n)$. If $\Yam_M(g)=-\infty$, then Theorem \ref{gtheorem}\eqref{gtheoreminftyb} will give equality. Moreover, if $(M,g)$ is almost homogeneous in the sense that there exists a bounded subset $U$ of $M$ such that for each $x\in M$ there is an isometry of $M$ with $f(x)\in U$, then $\Yam_M(g)=\Yaminfty_M(g)$: see \cite[Remark 14]{Grosse09}.

\smallskip
Equality in $-\norm{(\scal_g)_-}_{L^{n/2}(g)} \leq \Yam_M(g)$ can also occur. For instance, if $M$ is closed and $\scal_g$ is a nonpositive constant, then we have equality. For closed manifolds, $\scal_g$ being a nonpositive constant is the only possibility to get equality (this is easy to deduce from the Aubin--Schoen theorem \cite{LeeParker} which implies that the infimum in the definition of $\Yam_M(g)$ is achieved at some $v$). On noncompact manifolds equality holds also e.g.\ if $\Yam_M(g)=-\infty$.
\end{remarks}

While $\norm{(\scal_g)_-}_{L^{n/2}(g)} < \infty$ implies $\Yam_M(g) > -\infty$, the converse is in general not true: for instance, the $n$-dimensional hyperbolic space has Yamabe constant $\sigma(S^n)$, but satisfies $\norm{(\scal_g)_-}_{L^{n/2}(g)}=\infty$ because of its infinite volume and constant negative scalar curvature. That the two conditions are not equivalent should not be surprising: $\Yam_M(g)$ is a conformal invariant of $g$, but the $L^{n/2}(g)$-norm of $(\scal_g)_-$ is only invariant under rescalings of $g$ by constants. We expect that this is the only reason for the failure of equivalence:

\begin{conjecture}
Let $M$ be a nonempty manifold of dimension $n\geq3$, let $g\in\Metr(M)$. Then $\Yam_M(g)=-\infty$ holds if and only if $\bignorm{(\scal_{\bar{g}})_-}_{L^{n/2}(\bar{g})} = \infty$ holds for all metrics $\bar{g}$ in the conformal class of $g$.
\end{conjecture}

\noindent For instance, hyperbolic space is conformal to a subset of Euclidean space with $\norm{(\scal_g)_-}_{L^{n/2}(g)}=0$.

\begin{proof}[Proof of Theorem \ref{gtheorem}\eqref{gtheoreminftya}]
Let $\Yaminfty_M(g)<0$. Assume that $-\infty<\Yaminfty_M(g)$. Let $p=p_n$. We choose a compact exhaustion $(K'_i)_{i\geq1}$ of $M$ and define $K_0\define\leer$. We will construct recursively a compact exhaustion $(K_i)_{i\geq1}$ of $M$ and a sequence $(v_i)_{i\geq1}$ in $C_c^\infty(M,\R_{\geq0})\without\set{0}$ such that the properties
\begin{align*}
\supp(v_i) &\subseteq K_i\without K_{i-1} ,
&\Yamfunc_g(v_i) &\leq \Yaminfty_M(g) +\tfrac{1}{2^i} ,
&\norm{v_i}_{L^p(g)} &= 1
\end{align*}
hold for all $i\geq1$.

\smallskip
When $K_j$ and $v_j$ have already been constructed with these properties for all $j\in\N$ with $1\leq j<i$, we find $v_i$ as follows. Since $M\without K_{i-1}$ contains $M\without K'_j$ for all sufficiently large $j$, Fact \ref{increasing} yields $\Yam_{M\without K_{i-1}}(g) \leq \Yam_{M\without K'_j}(g)$ for all sufficiently large $j$. This implies $\Yam_{M\without K_{i-1}}(g) \leq \Yaminfty_M(g)$. Thus there exists a function $\tilde{v}_i\in C^\infty_c(M\without K_{i-1},\R_{\geq0})$ with $\Yamfunc_g(\tilde{v}_i) \leq \Yaminfty_M(g) +\tfrac{1}{2^i}$ and $\norm{\tilde{v}_i}_{L^p(g)} = 1$. We let $v_i\in C^\infty(M,\R_{\geq0})$ be the extension of $\tilde{v}_i$ with $\supp(v_i)=\supp(\tilde{v}_i)$ and define $K_i\define K'_{m(i)}$, where $m(0)\define0$ and $m(i)\define \min\set{j\in\N \suchthat j\geq i, \,j>m(i-1), \,\supp(v_i)\subseteq K'_j\without\mfbd K'_j}$. This completes the recursive definition of $(K_i)_{i\geq1}$ and $(v_i)_{i\geq1}$.

\smallskip
For each $i\geq1$, the properties $\supp(v_i)\subseteq K_i\without K_{i-1}$ and $\Yamfunc_g(v_i) \leq \Yaminfty_M(g) +\tfrac{1}{2^i}$ and $\norm{v_i}_{L^p(g)} = 1$ hold by construction. The sets $K_i$ form a compact exhaustion of $M$ because $(K'_i)_{i\geq1}$ is a compact exhaustion of $M$ (each $x\in M$ lies in some $K'_j$ and thus in $K_j$, and each $K_i$ lies in the interior of $K_{i+1}$ because $K'_{m(i)}$ lies in the interior of $K'_{m(i+1)}$). Thus $(K_i)_{i\geq1}$ and $(v_i)_{i\geq1}$ have the claimed properties.

\smallskip
For $j,k\in\N$ with $0\leq k<j$, we consider $w_{j,k}\define \sum_{i=k+1}^j v_i\restrict_{M\without K_k}\in C_c^\infty(M\without K_k,\R_{\geq0})$. Using that the supports of the functions $v_i$ are pairwise disjoint, we compute:
\[ \begin{split}
\Yam_{M\without {K}_k}(g) &\leq \Yamfunc_g(w_{j,k})
= \frac{\displaystyle\int_M \bigg(a_n\Bigabs{\textstyle\sum_{i=k+1}^j\diff v_i}_g^2 +\scal_g\Big({\textstyle\sum_{i=k+1}^j v_i}\Big)^2 \bigg)\,\diff\mu_g}{\displaystyle\bigg(\int_M \Big({\textstyle\sum_{i=k+1}^j v_i}\Big)^p\,\diff\mu_g\bigg)^{2/p}}\\
&=\frac{\displaystyle\sum_{i=k+1}^j \int_M\bigg(a_n\abs{\diff v_i}^2 +\scal_gv_i^2\bigg)\diff\mu_g }{\displaystyle\bigg(\sum_{i=k+1}^j\norm{v_i}_{L^p(g)}^p\bigg)^{{2/p}}}
= (j-k)^{-2/p}\sum_{i=k+1}^j\Yamfunc_g(v_i)\\
&\leq (j-k)^{-2/p}\left((j-k)\,\Yaminfty_M(g) +\sum_{i=k+1}^j\frac{1}{2^i}\right)\\
&\leq (j-k)^{2/n}\,\Yaminfty_M(g) +2 .
\end{split} \]
Since $\Yaminfty_M(g)<0$, this tends to $-\infty$ as $j\to\infty$. Thus we obtain $\Yam_{M\without K_k}(g) = -\infty$ for each $k$, in particular $\Yaminfty_M(g)=-\infty$, in contradiction to our assumption. Hence $\Yaminfty_M(g)=-\infty$.
\end{proof}

\begin{proof}[Proof of Theorem \ref{gtheorem}\eqref{gtheoreminftyb}]
Let $\Yam_M(g)=-\infty$. We argue by contradiction and assume $\Yaminfty_M(g)>-\infty$. Then there exists a compact subset $K_0\subset M$ with $\Yam_{M\without K_0}(g)>-\infty$. We choose a compact subset $K_1$ of $M$ whose interior contains $K_0$, and a smooth cutoff function $\eta\in C^\infty(M,[0,1])$ which is $1$ on a neighborhood of $K_0$ and vanishes on a neighborhood of the closure of $M\without K_1$. Theorem \ref{gtheorem}(\ref{gtheoremineq}) implies $\Yam_{K_1\without\mfbd K_1}(g) \geq -\norm{(\scal_g)_-}_{L^{n/2}(K_1;g)} > -\infty$. Let $p=p_n$, let $v\in C_c^\infty(M,\R_{\geq0})$ with $\int_M v^p\,\diff\mu_g=1$.

\smallskip
Since $\eta v\in C^\infty_c(K_1\without\mfbd K_1,\R_{\geq0})$ and $(1-\eta)v\in C^\infty_c(M\without K_0,\R_{\geq0})$, we obtain:
\[ \begin{split}
\Yamfunc_g(v) &= \int_M \Big( a_n\bigabs{\diff\big(\eta v+(1-\eta)v\big)}_g^2 +\scal_g\big(\eta v+(1-\eta)v\big)^2 \Big)\,\diff\mu_g\\
&= \int_M \Big( a_n\bigabs{\diff\big(\eta v\big)}_g^2 +\scal_g\big(\eta v\big)^2 \Big)\,\diff\mu_g
+\int_M \Big( a_n\bigabs{\diff\big((1-\eta)v\big)}_g^2 +\scal_g\big((1-\eta)v\big)^2 \Big)\,\diff\mu_g\\
&\mspace{20mu}+2\int_M \Big( a_n\bigeval{\diff\big(\eta v\big)}{\diff\big((1-\eta)v\big)}_g +\scal_g\,\eta\,(1-\eta)\,v^2 \Big)\,\diff\mu_g\\
&\geq \Yam_{K_1\without\mfbd K_1}(g)\left( \int_M \eta^p v^p\diff\mu_g\right)^{2/p}+\Yam_{M\without K_0}(g)\left( \int_M (1-\eta)^p v^p\diff\mu_g\right)^{2/p}\\
&\mspace{20mu}+2\int_M \Big( a_n\bigeval{\eta\,\diff v +v\,\diff\eta}{(1-\eta)\,\diff v -v\,\diff\eta}_g
+\scal_g\,\eta\,(1-\eta)\,v^2 \Big)\,\diff\mu_g\\[1ex]
&\geq \min\bigset{\Yam_{K_1\without\mfbd K_1}(g),\,0} +\min\bigset{\Yam_{M\without K_0}(g),\,0} -2\bignorm{\scal_g\,\eta\,(1-\eta)}_{L^{n/2}(g)}\\
&\mspace{20mu}
+2a_n\int_M \eta(1-\eta)\abs{\diff v}_g^2\,\diff\mu_g
+a_n\int_M \bigeval{2v\,\diff v}{(1-2\eta)\,\diff\eta}_g\,\diff\mu_g
-2a_n\int_M v^2\abs{\diff\eta}_g^2\,\diff\mu_g\\[1ex]
&\geq \min\bigset{\Yam_{K_1\without\mfbd K_1}(g),\,0} +\min\bigset{\Yam_{M\without K_0}(g),\,0} -2\bignorm{\scal_g}_{L^{n/2}(K_1\without K_0;g)}\\
&\mspace{20mu} +0
-a_n\int_M v^2\,\divergence_g\big((1-2\eta)\diff\eta\big)\,\diff\mu_g
-2a_n\bignorm{\abs{\diff\eta}_g^2}_{L^{n/2}(g)}\\[1ex]
&\geq \min\bigset{\Yam_{K_1\without\mfbd K_1}(g),\,0} +\min\bigset{\Yam_{M\without K_0}(g),\,0} -2\bignorm{\scal_g}_{L^{n/2}(K_1\without K_0;g)}\\
&\mspace{20mu}
-a_n\bignorm{\divergence_g\big((1-2\eta)\diff\eta\big)}_{L^{n/2}(g)}
-2a_n\bignorm{\abs{\diff\eta}_g^2}_{L^{n/2}(g)} .
\end{split} \]
This is a finite number independent of $v$. Hence $\Yam_M(g)>-\infty$, a contradiction.
\end{proof}

\section{Preparations for the fine continuity proofs}

\begin{lemma} \label{deltalemmafine}
Let $n\in\N$, let $(K_i)_{i\geq0}$ be a compact exhaustion of a Riemannian $n$-manifold $(M,g)$, let $(\eps_i)_{i\geq0}$ be a sequence of positive real numbers. Then there exists a function $\delta\in C^\infty(M,\R_{>0})$ which satisfies for every $i\geq0$ the inequalities $\delta\restrict_{M\without K_i}\leq\eps_i$ and
\begin{align*}
\bignorm{\delta}_{L^{n/2}(M\without K_i,g)} &\leq\eps_i,
&\bignorm{\diff \delta}_{L^{n}(M\without K_i,g)} &\leq\eps_i,\\
\bignorm{\delta\,\scal_g}_{L^{n/2}(M\without K_i,g)} &\leq\eps_i,
&\bignorm{\laplace_g\delta}_{L^{n/2}(M\without K_i,g)} &\leq\eps_i.
\end{align*}
\end{lemma}
\begin{proof}
We define $K'_{-1}\define\varnothing$ and $K'_i\define K_i\without(K_{i-1}\without\mfbd K_{i-1})$ for $i\geq0$. For each $i\geq0$, we choose a function $\beta_i\in C^\infty(K'_i,[0,1])$ which is constant $1$ near $K'_{i-1}\cap K'_i$ and is constant $0$ near $K'_i\cap K'_{i+1}$. We define recursively $\eps'_{-1}\define1$ and $\eps'_i\define\min\bigset{\tfrac{1}{2}\eps'_{i-1}, \eps_i} \in\R_{>0}$ for $i\geq0$. For all $j\geq i\geq0$, this implies $\eps'_j\leq 2^{-(j-i)}\eps_i$. Thus
\[
\forall i\geq0: \;\;\sum_{j>i}\eps'_j \leq \sum_{j>i}2^{-(j-i)}\eps_i = \eps_i .
\]
We let $\delta_{-1}\define1$ and, for all $i\geq0$,
\[
\delta_i \define \min\Biggset{\delta_{i-1},\, \eps_i,\,
\frac{\eps'_i}{\norm{1}_{L^{n/2}(K'_i,g)} +\norm{\scal_g}_{L^{n/2}(K'_i,g)} +\norm{\diff\beta_i}_{L^n(K'_i,g)} +\norm{\laplace_g\beta_i}_{L^{n/2}(K'_i,g)} } } >0 .
\]
The function $\delta\in C^\infty(M,\R)$ given by $\delta\restrict_{K'_i} = (\delta_i-\delta_{i+1})\beta_i +\delta_{i+1}$ is positive because $(\delta_i)_{i\geq0}$ is a monotonically decreasing sequence of positive numbers. It satisfies $\delta\restrict_{M\without K_i}\leq\eps_i$ for every $i\geq0$, because $(\delta_i)_{i\geq0}$ is monotonically decreasing with $\delta\restrict_{K'_i} \leq \delta_i \leq \eps_i$. Since $M\without K_i\subseteq \bigcup_{j>i}K'_j$ holds for every $i\geq0$, we obtain for $i\geq0$:
\begin{align*}
\bignorm{\delta}_{L^{n/2}(M\without K_i,g)} &\leq \sum_{j>i}\bignorm{\delta}_{L^{n/2}(K'_j,g)}
\leq \sum_{j>i}\delta_j\bignorm{1}_{L^{n/2}(K'_j,g)} \leq \sum_{j>i}\eps'_j \leq \eps_i ,\\
\bignorm{\delta\,\scal_g}_{L^{n/2}(M\without K_i,g)} &\leq \sum_{j>i}\bignorm{\delta\,\scal_g}_{L^{n/2}(K'_j,g)}
\leq \sum_{j>i}\delta_j\bignorm{\scal_g}_{L^{n/2}(K'_j,g)} \leq \sum_{j>i}\eps'_j \leq \eps_i ,\\
\bignorm{\diff\delta}_{L^{n}(M\without K_i,g)} &\leq \sum_{j>i}\bignorm{\diff\delta}_{L^{n}(K'_j,g)}
\leq \sum_{j>i}\delta_j\bignorm{\diff\beta_j}_{L^{n}(K'_j,g)} \leq \sum_{j>i}\eps'_j \leq \eps_i ,\\
\bignorm{\laplace_g\delta}_{L^{n/2}(M\without K_i,g)} &\leq \sum_{j>i}\bignorm{\laplace_g\delta}_{L^{n/2}(K'_j,g)}
\leq \sum_{j>i}\delta_j\bignorm{\laplace_g\beta_j}_{L^{n/2}(K'_j,g)} \leq \sum_{j>i}\eps'_j \leq \eps_i .\qedhere
\end{align*}
\end{proof}

\begin{lemma} \label{topologylemmastrong}
Let $n\in\N$, let $(K_i)_{i\in\N}$ be a compact exhaustion of a Riemannian $n$-manifold $(M,g)$, let $(\eps_i)_{i\in\N}$ be a sequence of positive real numbers. Then there exist a fine $C^2$-neighborhood $\mathcal{U}$ of $g$ and a function $\delta\in C^\infty(M,\R_{>0})$ such that the following conditions hold for all $h\in\mathcal{U}$:
\begin{enumerate}
\item $\forall i\in\N\colon$ \;\;$\delta\restrict_{M\without K_i}\leq\eps_i$ \;\;and\;\; $\norm{\delta}_{L^{n/2}(M\without K_i,g)} \leq\eps_i$ \;\;and\;\; $\norm{\delta\,\scal_g}_{L^{n/2}(M\without K_i,g)} \leq\eps_i$.
\item $\forall x\in M\colon \forall \alpha\in T^\ast_xM\colon \bigabs{\abs{\alpha}_h^2 -\abs{\alpha}_g^2} \leq \delta(x)\abs{\alpha}_g^2$.
\item $\abs{\scal_g-\scal_h} \leq \delta$.
\item $\Bigabs{1-\F{h}{g}}\leq\delta$.
\item $\forall i\in\N\colon \Bignorm{\diff\Big(\big((1-\delta)\F{h}{g}\big)^{1/2}\Big)}_{L^n(M\without K_i,g)}^2 \leq \eps_i \;\;\text{\;\;and\;\;} \;\;\Bignorm{\laplace_g\Big((1-\delta)\F{h}{g}\Big)}_{L^{n/2}(M\without K_i,g)} \leq \eps_i$.
\end{enumerate}
\end{lemma}
\begin{proof}
For each $i\in\N$, we choose $\tilde{\eps}_i\in\R_{>0}$ so small that
\begin{align*}
\tilde{\eps}_i &\leq \tfrac{1}{2} ,
&\sqrt{\tilde{\eps}_i} +\tilde{\eps}_i &\leq \sqrt{\eps_i},
&3\tilde{\eps}_i +2\tilde{\eps}_i^{3/2} &\leq \eps_i .
\end{align*}
We apply Lemma \ref{deltalemmafine} to the sequence $(\tilde{\eps}_i)_{i\in\N}$ and obtain a function $\delta\in C^\infty(M,\R_{>0})$ with the properties stated in Lemma \ref{deltalemmafine}, but with $\tilde{\eps}_i$ instead of $\eps_i$. Then condition (1) holds, because $\forall i\in\N\colon \tilde{\eps}_i\leq\eps_i$. The Examples \ref{fineexamples} imply that $g$ has a fine $C^2$-neighborhood $\mathcal{U}$ such that every $h\in\mathcal{U}$ satisfies
\begin{enumerate} \renewcommand{\labelenumi}{(\alph{enumi})}
\item $\abs{\scal_h-\scal_g} < \delta$;
\item $\forall x\in M\colon \max\bigset{\bigabs{\abs{\alpha}_h^2-1} \bigsuchthat \alpha\in T_x^\ast M, \,\abs{\alpha}_g=1} < \delta(x)$;
\item $\bigabs{\F{h}{g}-1} < \delta$;
\item $\bigabs{\diff\big(\F{h}{g}\big)}_g < \delta$;
\item $\bigabs{\laplace_g\big(\F{h}{g}\big)} < \delta$.
\end{enumerate}
Property (b) yields condition (2): that's because $\bigabs{\abs{\alpha}_h^2 -\abs{\alpha}_g^2} \leq \delta(x)\abs{\alpha}_g^2$ holds for $\alpha=0$, and because for $\alpha\in T_x^\ast M\without\set{0}$, \,$\beta\define\alpha/\abs{\alpha}_g$ satisfies $\abs{\beta}_g=1$ and thus $\bigabs{\abs{\beta}_h^2-1} \leq \delta(x)$, which implies that $\alpha$ satisfies $\bigabs{\abs{\alpha}_h^2 -\abs{\alpha}_g^2} \leq \delta(x)\abs{\alpha}_g^2$. The properties (a) and (c) yield (3) and (4), respectively. It remains to verify (5). Using $\tfrac{1}{2} \leq 1-\tilde{\eps}_i \leq 1-\delta\restrict_{M\without K_i} \leq 1$ and (c) and (d), we obtain on $M\without K_i$:
\[ \begin{split}
\Bigabs{\diff\Big(\big((1-\delta)\F{h}{g}\big)^{1/2}\Big)}_g
&= \Biggabs{ \frac{(1-\delta)\diff\big(\F{h}{g}\big) -\F{h}{g}\,\diff\delta}{2\big((1-\delta)\F{h}{g}\big)^{1/2}} }_g\\
&\leq \Biggabs{ \frac{\diff\big(\F{h}{g}\big)}{\sqrt{2}\big(\F{h}{g}\big)^{1/2}} }_g
+\Biggabs{ \frac{\big(\F{h}{g}\big)^{1/2}\,\diff\delta}{\sqrt{2}} }_g\\
&\leq \frac{\delta}{\sqrt{2(1-\delta)}}
+\sqrt{\frac{1+\delta}{2}}\abs{\diff\delta}_g\\
&\leq \delta +\abs{\diff\delta}_g .
\end{split} \]
Hence, because of $\delta\leq1$ and the properties stated in Lemma \ref{deltalemmafine}:
\begin{multline*}
\forall i\in\N\colon
\Bignorm{\diff\Big(\big((1-\delta)\F{h}{g}\big)^{1/2}\Big)}_{L^n(M\without K_i,g)}
\leq \bigg(\int_{M\without K_i}\delta^n\,\diff\mu_g\bigg)^{1/n}
+\norm{\diff\delta}_{L^n(M\without K_i;g)}\\
\leq \bigg(\int_{M\without K_i}\delta^{n/2}\,\diff\mu_g\bigg)^{1/n} +\tilde{\eps}_i
= \norm{\delta}_{L^{n/2}(M\without K_i;g)}^{1/2} +\tilde{\eps}_i
\leq \sqrt{\tilde{\eps}_i} +\tilde{\eps}_i \leq \sqrt{\eps_i} .
\end{multline*}
Thus the first inequality in (5) holds. Similarly we get from (c), (d), (e):
\[ \begin{split}
\Bigabs{\laplace_g\Big((1-\delta)\F{h}{g}\Big)}
&= \Bigabs{ (1-\delta)\laplace_g\big(\F{h}{g}\big) -\F{h}{g}\laplace_g\delta -2\bigeval{\diff\delta}{\diff\big(\F{h}{g}\big)}_g }
\leq \delta
+2\bigabs{\laplace_g\delta}
+2\delta\,\abs{\diff\delta}_g ;
\end{split} \]
hence
\[ \begin{split}
\Bignorm{\laplace_g\Big((1-\delta)\F{h}{g}\Big)}_{L^{n/2}(M\without K_i,g)}
&\leq \bignorm{\delta}_{L^{n/2}(M\without K_i,g)}
+2\bignorm{\laplace_g\delta}_{L^{n/2}(M\without K_i,g)}\\
&\mspace{20mu}
+2\bignorm{\delta}_{L^n(M\without K_i,g)}\,\bignorm{\diff\delta}_{L^n(M\without K_i,g)}\\
&\leq \tilde{\eps}_i +2\tilde{\eps}_i +2\sqrt{\tilde{\eps}_i}\,\tilde{\eps}_i\\
&\leq \eps_i .
\end{split} \]
Thus also the second inequality in (5) holds.
\end{proof}

\begin{corollary} \label{topologylemma}
Let $n\in\N$, let $(M,g)$ be a Riemannian $n$-manifold, let $\eps\in\R_{>0}$. Then there exist a fine $C^2$-neighborhood $\mathcal{U}$ of $g$ and a function $\delta\in C^\infty(M,\R_{>0})$ such that the following conditions hold for all $h\in\mathcal{U}$:
\begin{enumerate}
\item $\delta\leq\eps$ \;\;and\;\; $\norm{\delta}_{L^{n/2}(g)} \leq\eps$ \;\;and\;\; $\norm{\delta\,\scal_g}_{L^{n/2}(M,g)} \leq\eps$.
\item $\forall x\in M\colon \forall \alpha\in T^\ast_xM\colon \bigabs{\abs{\alpha}_h^2 -\abs{\alpha}_g^2} \leq \delta(x)\abs{\alpha}_g^2$.
\item $\abs{\scal_g-\scal_h} \leq \delta$.
\item $\Bigabs{1-\F{h}{g}} \leq \delta$.
\item $\Bignorm{\diff\Big(\big((1-\delta)\F{h}{g}\big)^{1/2}\Big)}_{L^n(g)}^2 \leq \eps$ \;\;and\;\; $\Bignorm{\laplace_g\Big((1-\delta)\F{h}{g}\Big)}_{L^{n/2}(g)} \leq \eps$.
\end{enumerate}
\end{corollary}
\begin{proof}
We choose any compact exhaustion $(K_i)_{i\geq0}$ of $M$ with $K_0=\leer$ and consider the sequence $(\eps_i)_{i\geq0}$ with $\forall i:\eps_i=\eps$. The claim of the Corollary is the $i=0$ statement of Lemma \ref{topologylemmastrong}.
\end{proof}

\section{Fine continuity: proofs of the theorems \ref{finecontinuity} and \ref{locallyconstant} and \ref{imagetheoremclopen}}

\begin{proof}[Proof of Theorem \ref{finecontinuity}]
By Corollary \ref{upper}, with respect to the fine $C^2$-topology $\Yam_M$ is upper semicontinuous, and continuous at every $g$ with $\Yam_M(g)=-\infty$. It remains to prove lower semicontinuity at each $g\in\Metr(M)$ with $\Yam_M(g)>-\infty$. For such a $g$, let $\eps_0\in\R_{>0}$.

\smallskip
Let $p=\frac{2n}{n-2}$. We choose $\eps\in\ooi{0}{1}$ so small that
\begin{equation} \label{3cond} \begin{split}
\eps(1-\eps)^{-2/p}\left(\frac{3a_n}{2}+4\right) &\leq \eps_0 ,\\
\Big((1-\eps)^{-2/p}-1\Big)\abs{\Yam_M(g)} &\leq \eps_0 ,\\
\Big(1-(1+\eps)^{-2/p}(1-\eps)^2\Big)\abs{\Yam_M(g)} &\leq \eps_0 .
\end{split} \end{equation}
There exist a fine $C^2$-neighborhood $\mathcal{U}$ of $g$ and a function $\delta\in C^\infty(M,\R_{>0})$ with the properties stated in Corollary \ref{topologylemma}. For every $h\in\mathcal{U}$ and every $v\in C_c^\infty(M,\R_{\geq0})$ with $\int_M v^p\,\diff\mu_h = 1$, we have to estimate $\Yamfunc_h(v)$ from below.

\smallskip
Since $\delta\leq\eps<1$ by \ref{topologylemma}(1), we can consider $w = \big(\F{h}{g}(1-\delta)\big)^{1/2}\,v \in C_c^\infty(M,\R_{\geq0})\without\set{0}$. With \ref{topologylemma}(2,3), we obtain:
\[ \begin{split}
\Yamfunc_h(v)
&= a_n\int_M\abs{\diff v}_h^2\,\diff\mu_h
+\int_M\scal_h\,v^2\,\diff\mu_h\\
&\geq a_n\int_M(1-\delta)\,\F{h}{g}\,\abs{\diff v}_g^2\,\diff\mu_g
+\int_M \F{h}{g}(\scal_g-\delta)\,v^2\,\diff\mu_g\\
&\geq a_n\int_M\abs{\diff w}_g^2\,\diff\mu_g
-a_n\int_M v^2\,\Bigabs{\diff\Big(\big(\F{h}{g}(1-\delta)\big)^{1/2}\Big)}_g^2\,\diff\mu_g\\
&\mspace{20mu}-2a_n\int_M \Bigeval{v\,\diff v}{\, \big(\F{h}{g}(1-\delta)\big)^{1/2}\, \diff\Big(\big(\F{h}{g}(1-\delta)\big)^{1/2}\Big)}_g\,\diff\mu_g\\
&\mspace{20mu}+\int_M\scal_g(1-\delta)\,\F{h}{g}\,v^2\,\diff\mu_g
-\int_M\scal_g(1-\delta)\,\F{h}{g}\,v^2\,\diff\mu_g
+\int_M\F{h}{g}\,(\scal_g-\delta)\,v^2\,\diff\mu_g\\
&= \Yamfunc_g(w)\norm{w}_{L^p(g)}^2
-a_n\int_M v^2\,\Bigabs{\diff\Big(\big(\F{h}{g}(1-\delta)\big)^{1/2}\Big)}_g^2\,\diff\mu_g\\
&\mspace{20mu}-\frac{a_n}{2}\int_M v^2\,\laplace_g\big(\F{h}{g}(1-\delta)\big)\,\diff\mu_g
+\int_M\F{h}{g}\,\delta\,\scal_g\,v^2\,\diff\mu_g
-\int_M\F{h}{g}\,\delta\,v^2\,\diff\mu_g .
\end{split} \]
Corollary \ref{topologylemma}(4) yields $1-\eps\leq \F{h}{g}\leq 1+\eps$. Thus $\norm{v}_{L^p(g)} \leq (1-\eps)^{-1/p}\,\norm{v}_{L^p(h)} = (1-\eps)^{-1/p}$. Using this estimate and \ref{topologylemma}(1,5) and
\[
\bignorm{\F{h}{g}\,\delta\,\scal_g}_{L^{n/2}(g)} \leq \bignorm{\F{h}{g}}_{L^\infty(g)}\,\bignorm{\delta\,\scal_g}_{L^{n/2}(g)} \leq 2\,\bignorm{\delta\,\scal_g}_{L^{n/2}(g)} ,
\]
we obtain:
\[ \begin{split}
\Yamfunc_h(v)
&\geq \Yam_M(g)\norm{w}_{L^p(g)}^2
-a_n(1-\eps)^{-2/p}\, \Bignorm{\diff\Big(\big(\F{h}{g}(1-\delta)\big)^{1/2}\Big)}_{L^n(g)}^2\\
&\mspace{20mu}
-\frac{a_n}{2}(1-\eps)^{-2/p}\,\bignorm{\laplace\big(\F{h}{g}(1-\delta)\big)}_{L^{n/2}(g)}
-2(1-\eps)^{-2/p}\,\norm{\delta\,\scal_g}_{L^{n/2}(g)}\\
&\mspace{20mu}-2(1-\eps)^{-2/p}\,\norm{\delta}_{L^{n/2}(g)}\\
&\geq \Yam_M(g)\norm{w}_{L^p(g)}^2
-\eps(1-\eps)^{-2/p}\left(\frac{3a_n}{2}+4\right)\\
&\geq \Yam_M(g)\norm{w}_{L^p(g)}^2 -\eps_0.
\end{split} \]
Since $w^2 = (1-\delta)\F{h}{g}\,v^2 \leq (1-\delta)(1+\delta)v^2 \leq v^2$ by Corollary \ref{topologylemma}(1,4), we have
\[
\norm{w}_{L^p(g)}^2 \leq (1-\eps)^{-2/p}\,\norm{w}_{L^p(h)}^2
\leq (1-\eps)^{-2/p}\,\norm{v}_{L^p(h)}^2 = (1-\eps)^{-2/p} .
\]
On the other hand, $w^2 = (1-\delta)\F{h}{g}\,v^2 \geq (1-\delta)^2v^2 \geq (1-\eps)^2v^2$ yields
\[
\norm{w}_{L^p(g)}^2 \geq (1+\eps)^{-2/p}\norm{w}_{L^p(h)}^2 \geq (1+\eps)^{-2/p}(1-\eps)^2\norm{v}_{L^p(h)}^2 = (1+\eps)^{-2/p}(1-\eps)^2 .
\]
Therefore we obtain from \eqref{3cond}:
\[
\Yamfunc_h(v) \geq
\left.\begin{cases}
(1-\eps)^{-2/p}\,\Yam_M(g)-\eps_0 & \text{if $\Yam_M(g)\leq 0$}\\
(1+\eps)^{-2/p}(1-\eps)^2\,\Yam_M(g)-\eps_0 & \text{if $\Yam_M(g)> 0$}
\end{cases}\right\}
\geq \Yam_M(g)-2\eps_0 .
\]
This holds for all $v\in C_c^\infty(M,\R_{\geq0})$ with $\int_M v^p\,\diff\mu_h = 1$ and thus for all $v\in C_c^\infty(M,\R_{\geq0})\without\set{0}$. Taking the infimum over all such $v$ yields $\Yam_M(h)\geq \Yam_M(g)-2\eps_0$. Since for every $\eps_0\in\R_{>0}$ there exists a neighborhood $\mathcal{U}$ of $g$ such that this is true for all $h\in\mathcal{U}$, the map $\Yam_M$ is lower semicontinuous at $g$.
\end{proof}

Following essentially the same proof we would see that also $\Yaminfty_M$ is continuous with respect to the fine $C^2$-topology. But we will show even more: that $\Yaminfty_M$ is locally constant.

\begin{proof}[Proof of Theorem \ref{locallyconstant}]
We have to show that each $g\in \Metr(M)$ has a fine $C^2$-neighborhood on which $\Yaminfty_M$ is constant. Let $(K_i)_{i\in\N}$ be a compact exhaustion of $M$. We first study the case where $\Yam_{M\without K_{i_0}}(g) > -\infty$ holds for some $i_0\in\N$. By Fact \ref{increasing}, $\Yam_{M\without K_i}(g) > -\infty$ holds then for all $i\geq i_0$.

\smallskip
For $i>i_0$, there exists a $v_i\in C^\infty(M,\R_{\geq0})$ which has compact support in $M\without K_i$ and satisfies $\norm{v_i}_{L^p(g)} = 1$ and $\Yamfunc_g(v_i) \leq \Yam_{M\without K_i}(g) +i^{-1}$. For $A_i\define \int_M(\scal_g)_-\,v_i^2\,\diff\mu_g$ (which is a finite number because $v_i$ has compact support), we choose $\eps_i\in\R_{>0}$ so small that
\begin{align} \label{2cond} \begin{split}
(1-{\eps_i})^{-2/p}\Big((1+\eps_i)^2i^{-1} +(\eps_i^2+3\eps_i)A_i +(1+\eps_i)\eps_i\Big )
&\leq 2i^{-1} ,\\
(1-\eps_i)^{-2/p}\,(1+\eps_i)^2\,\abs{\Yam_{M\without K_i}(g)}
&\leq \abs{\Yam_{M\without K_i}(g)} +i^{-1}
\end{split} \end{align}
and
\begin{equation} \label{3condb} \begin{split}
\eps_i(1-\eps_i)^{-2/p}\left(\frac{3a_n}{2}+4\right) &\leq i^{-1} ,\\
\Big((1-\eps_i)^{-2/p}-1\Big)\abs{\Yam_{M\without K_i}(g)} &\leq i^{-1} ,\\
\Big(1-(1+\eps_i)^{-2/p}(1-\eps_i)^2\Big)\abs{\Yam_{M\without K_i}(g)} &\leq i^{-1} .
\end{split} \end{equation}

We choose $\eps_i\in\R_{>0}$ arbitrarily for $i\leq i_0$. For the resulting sequence $(\eps_i)_{i\in\N}$, there exist a function $\delta\in C^\infty(M,\R_{>0})$ and a fine $C^2$-neighborhood $\mathcal{U}$ of $g$ with the properties stated in Lemma \ref{topologylemmastrong}. We obtain for every $h\in\mathcal{U}$ and every $i>i_0$:
\[ \begin{split}
\Yam_{M\without K_i}(h) &\leq \Yamfunc_h(v_i)
= \norm{v_i}_{L^p(h)}^{-2} \left( a_n\int_{M\without K_i}\abs{\diff v_i}_h^2\,\diff\mu_h
+\int_{M\without K_i}\scal_h\,v_i^2\,\diff\mu_h \right)\\
&\leq \norm{v_i}_{L^p(h)}^{-2} \left(
a_n\int_{M\without K_i}(1+\eps_i)^2\,\abs{\diff v_i}_g^2\,\diff\mu_g
+\int_{M\without K_i}(\delta+\scal_g)\,v_i^2\,\F{h}{g}\,\diff\mu_g \right)\\
&\leq \norm{v_i}_{L^p(h)}^{-2} \left( (1+\eps_i)^2\,\Yamfunc_g(v_i)
-(1+\eps_i)^2\int_{M\without K_i}\scal_g\,v_i^2\,\diff\mu_g \right.\\[-1ex]
&\mspace{110mu}\left. +\int_{M\without K_i}\scal_g\,v_i^2\,\F{h}{g}\,\diff\mu_g
+(1+\eps_i)\norm{\delta}_{L^{n/2}(M\without K_i;g)} \right).
\end{split} \]
Using $-\eps_i^2 -3\eps_i = (1-\eps_i) -(1+\eps_i)^2 \leq \F{h}{g}\big|_{M\without K_i} -(1+\eps_i)^2
\leq (1+\eps_i) -(1+\eps_i)^2 < 0$, we get
\[ \begin{split}
\Yam_{M\without K_i}(h)
&\leq \norm{v_i}_{L^p(h)}^{-2}\left( (1+{\eps_i})^2 \Yamfunc_g(v_i)
+\int_{M\without K_i} \Big(\F{h}{g}-(1+\eps_i)^2\Big)\,\scal_g\,v_i^2\,\diff\mu_g
+(1+{\eps_i})\eps_i\right)\\
&\leq \norm{v_i}_{L^p(h)}^{-2}\left((1+{\eps_i})^2 \Yamfunc_g(v_i)
+\int_{M\without K_i} \Big((1+\eps_i)^2-\F{h}{g}\Big)(\scal_g)_-\,v_i^2\,\diff\mu_g
+(1+\eps_i)\eps_i\right)\\
&\leq \norm{v_i}_{L^p(h)}^{-2}\left((1+\eps_i)^2 \Yamfunc_g(v_i) +(\eps_i^2+3{\eps_i})\int_{M\without K_i}(\scal_g)_-\,v_i^2\,\diff\mu_g
+(1+\eps_i)\eps_i\right)\\
&\leq \norm{v_i}_{L^p(h)}^{-2}\Big( (1+\eps_i)^2 \big(\Yam_{M\without K_i}(g)+i^{-1}\big) +(\eps_i^2+3\eps_i)A_i +(1+\eps_i)\eps_i \Big) .
\end{split} \]
Since $(1-\eps_i)^{2/p} \leq \norm{v_i}_{L^p(M\without K_i;h)}^{2}\leq (1+\eps_i)^{2/p}$ and $A_i\geq0$ and $2-\frac{2}{p}>0$, we obtain from \eqref{2cond} in the case $\Yam_{M\without K_i}(g)<0$:
\[ \begin{split}
\Yam_{M\without K_i}(h)
&\leq \frac{(1+\eps_i)^2i^{-1} +(\eps_i^2+3\eps_i)A_i +(1+\eps_i)\eps_i}{(1-\eps_i)^{2/p}} +\frac{(1+\eps_i)^2\,\Yam_{M\without K_i}(g)}{(1+\eps_i)^{2/p}}
\leq \Yam_{M\without K_i}(g) +\frac{2}{i} ;
\end{split} \]
and in the case $\Yam_{M\without K_i}(g)\geq0$:
\[ \begin{split}
\Yam_{M\without K_i}(h)
&\leq \frac{(1+\eps_i)^2i^{-1} +(\eps_i^2+3\eps_i)A_i +(1+\eps_i)\eps_i}{(1-\eps_i)^{2/p}} +\frac{(1+\eps_i)^2\,\Yam_{M\without K_i}(g)}{(1-\eps_i)^{2/p}}
\leq \Yam_{M\without K_i}(g) +\frac{3}{i} .
\end{split} \]
As this holds for every $i>i_0$, we have $\Yaminfty_M(h)\leq\Yaminfty_M(g)$ for all $h\in\mathcal{U}$.

\medskip
The proof of $\Yaminfty_M(h)\geq \Yaminfty_M(g)$ works now almost exactly as the estimates in the lower semicontinuity part of the proof of Theorem \ref{finecontinuity}: We replace every $\eps$ by $\eps_i$, replace every $M$ by $M\without K_i$, replace every $\eps_0$ by $i^{-1}$, consider test functions $v\in C^\infty_c(M\without K_i,\R_{\geq0})\without\set{0}$ instead of $v\in C^\infty_c(M,\R_{\geq0})\without\set{0}$, define $w\define\big(\F{h}{g}(1-\delta)\big)^{1/2}v$ as before, use \eqref{3condb} instead of \eqref{3cond}, and apply the properties of $\delta$ and $\mathcal{U}$ from Lemma \ref{topologylemmastrong} instead of Corollary \ref{topologylemma}. For each $i>i_0$, we obtain in this way $\Yamfunc_h(v)\geq \Yam_{M\without K_i}(g) -2i^{-1}$ for all $v\in C^\infty_c(M\without K_i,\R_{\geq0})\without\set{0}$, hence $\Yam_{M\without K_i}(h) \geq \Yam_{M\without K_i}(g) -2i^{-1}$. This implies $\Yaminfty_M(h)\geq\Yaminfty_M(g)$.

\smallskip
Thus each $g\in \Metr(M)$ with $\Yam_{M\without K_i}(g)>-\infty$ has a fine $C^2$-neighborhood $\mathcal{U}$ on which $\Yaminfty_M$ is constant.

\medskip
It remains to consider the case where $\Yam_{M\without K_i}(g)=-\infty$ for all $i\in\N$. For every $i>0$, there exists a function $v_i\in C_c^\infty(M\without K_i,\R_{\geq0})$ with $\norm{v_i}_{L^p(g)}=1$ and $\Yamfunc_g(v_i)\leq-i$. For $A_i \define \int_M(\scal_g)_-\,v_i^2\,\diff\mu_g$, we choose $\eps_i\in\R_{>0}$ so small that the first inequality of \eqref{2cond} is valid. There exist a fine $C^2$-neighborhood $\mathcal{U}$ of $g$ and a function $\delta\in C^\infty(M,\R_{>0})$ with the properties stated in Lemma \ref{topologylemmastrong}. The same estimate as above yields
\[ \begin{split}
\Yam_{M\without K_i}(h)
&\leq \norm{v_i}_{L^p(h)}^{-2}\left((1+\eps_i)^2 \Yamfunc_g(v_i) +(\eps_i^2+3{\eps_i})\int_{M\without K_i}(\scal_g)_-\,v_i^2\,\diff\mu_g
+(1+\eps_i)\eps_i\right)\\
&\leq \frac{(1+\eps_i)^2 \Yamfunc_g(v_i)}{(1+\eps_i)^{2/p}} +\frac{(1+\eps_i^2)i^{-1} +(\eps_i^2+3\eps_i)A_i +(1+\eps_i)\eps_i}{(1-\eps_i)^{2/p}}\\
&\leq -i +2i^{-1}.
\end{split} \]
Since this holds for all $i>0$, we obtain $\Yaminfty_M(h)=-\infty$ for all $h\in\mathcal{U}$.

\smallskip
Hence, in each case, every $g\in\mathcal{U}$ has a fine $C^2$-neighborhood on which $\Yaminfty_M$ is constant.
\end{proof}

\begin{proof}[Proof of Theorem \ref{imagetheoremclopen}]
$\Yam_M^{-1}(\set{-\infty})$ is fine $C^2$-closed in $\Metr(M)$ because of Theorem \ref{finecontinuity}. Theorem \ref{gtheorem}(\ref{gtheoremineq},\ref{gtheoreminftyb}) tells us that $\Yam_M^{-1}(\set{-\infty})$ is equal to $\Yaminfty_M^{-1}(\set{-\infty})$. Theorem \ref{locallyconstant} implies that $\Yaminfty_M^{-1}(\set{-\infty})$ is fine $C^2$-open in $\Metr(M)$.
\end{proof}

\section{Proof of Theorem \ref{imagetheorem}}

\begin{lemma} \label{cylinder}
Let $(M,h)$ be a closed Riemannian manifold of dimension $n\geq2$. If $n\geq3$, assume $\Yam_M(h)<0$. If $n=2$, assume that $M$ has negative Euler characteristic. Then there exists an $i_h\in\R_{>0}$ such that for every $i\in\coi{i_h}{\infty}$ and every Riemannian metric $g$ on $M\times\R$ which coincides with $h+\diff t^2$ on $M\times[0,3i]$, the inequality $\Yam_{M\times\R}(g) \leq -i^{1/(n+1)}$ holds. In particular $\Yam_{M\times \R}(h+\diff t^2) = -\infty$.
\end{lemma}

\begin{proof}
Let $p=\frac{2(n+1)}{(n+1)-2}$. If $\dim(M)\geq3$, then by the solution of the Yamabe problem on closed manifolds, there is a function $w\in C^\infty(M,\R_{>0})$ with $\Yamfunc_h(w) = \Yam_M(h)$ and $\norm{w}_{L^q(h)}=1$, where $q=\frac{2n}{n-2}$. Since $\Yam_M(h)<0$, there exists a number $i_h>0$ such that
\[
\frac{8a_n\,\norm{w}_{L^2(h)}^2}{i^{1+2/p}\,\norm{w}_{L^p(h)}^2} +i^{2/(n+1)}\,\frac{\Yam_M(h)}{3^{2/p}\,\norm{w}_{L^p(h)}^2}
\leq -i^{1/(n+1)}
\]
holds for all $i\in\coi{i_h}{\infty}$. For such an $i$, let $g$ be a Riemannian metric on $M\times\R$ which coincides with $h+\diff t^2$ on $M\times[0,3i]$.

\smallskip
We choose a function $u_i\in C^\infty(\R,[0,1])$ with $\supp(u_i)\subset[0,3i]$ and $u_i\restrict_{[i,2i]}\equiv1$ and $\abs{u_i'}\leq\frac{2}{i}$. Then $i \leq \norm{u_i}_{L^2(\R)}^2$ and $\norm{u_i'}_{L^2(\R)}^2 \leq \frac{4}{i^2}\cdot2i = \frac{8}{i}$ and $i^{2/p} \leq \norm{u_i}_{L^p(\R)}^2 \leq (3i)^{2/p}$, hence
\begin{align*}
\frac{\norm{u_i'}_{L^2(\R)}^2}{\norm{u_i}_{L^p(\R)}^2} &\leq \frac{8}{i^{1+2/p}},
&\frac{\norm{u_i}_{L^2(\R)}^2}{\norm{u_i}_{L^p(\R)}^2} &\geq \frac{i}{(3i)^{2/p}} = \frac{i^{2/(n+1)}}{3^{2/p}}.
\end{align*}
We consider the function $v_i\in C^\infty(M\times\R,\R_{\geq0})\without\set{0}$ defined by $v_i(x,t)\define w(x)u_i(t)$.

\smallskip
Since $\scal_g(x,t) = \scal_h(x)$ and $\diff_{(x,t)}v_i(z,1) = u_i(t)\diff_xw(z) +w(x)u_i'(t)$ for all $(x,t)\in M\times[0,3i]$ and $z\in T_xM$, and since $a_{n+1} = \frac{4n}{n-1} < \frac{4(n-1)}{n-2} = a_n$ and $\Yam_M(h)<0$, we obtain:
\[ \begin{split}
\Yamfunc_g(v_i)
&= \frac{\displaystyle \int_{M\times\R}\bigg(a_{n+1}\,\abs{\diff v_i}_g^2 +\scal_g\,v_i^2\bigg)\diff\mu_g}{\displaystyle\bigg(\int_{M\times\R}v_i^p\,\diff\mu_g\bigg)^{2/p}}
\leq \frac{\displaystyle \int_{M\times[0,3i]}\bigg(a_n\abs{\diff v_i}_g^2 +\scal_g\,v_i^2\bigg)\diff\mu_g}{\displaystyle\bigg(\int_{M\times[0,3i]}v_i^p\,\diff\mu_g\bigg)^{2/p}}
\\
&= \frac{\displaystyle a_n\int_Mw^2\,\diff\mu_h \cdot \int_{[0,3i]}(u_i')^2\,\diff t
+\int_M\bigg(a_n\abs{\diff w}_h^2 +\scal_h\,w^2\bigg)\diff\mu_h \cdot\int_{[0,3i]}u_i^2\,\diff t
}{\displaystyle\bigg(\int_Mw^p\,\diff\mu_h\cdot \int_{[0,3i]}u_i^p\,\diff t\bigg)^{2/p}}\\
&= \frac{a_n\,\norm{w}_{L^2(h)}^2\,\norm{u_i'}_{L^2(\R)}^2 +\Yam_M(h)\norm{u_i}_{L^2(\R)}^2
}{\norm{w}_{L^p(h)}^2 \, \norm{u_i}_{L^p(\R)}^2}\\
&\leq \frac{8a_n\,\norm{w}_{L^2(h)}^2}{i^{1+2/p}\,\norm{w}_{L^p(h)}^2} +i^{2/(n+1)}\,\frac{\Yam_M(h)}{3^{2/p}\,\norm{w}_{L^p(h)}^2}
\leq -i^{1/(n+1)} .
\end{split} \]
Thus $\Yam_{M\times\R}(g) \leq \Yamfunc_g(v_i) \leq -i^{1/(n+1)}$ and $\Yam_{M\times \R}(h+\diff t^2) \leq \inf\set{-i^{1/(n+1)} \suchthat i\in\coi{i_h}{\infty}} = -\infty$.

\smallskip
It remains to prove the case where $M$ is a closed $2$-manifold with $\chi(M)<0$. There exists an $i_h\in\R_{>0}$ with
\[
\forall i\in\coi{i_h}{\infty}\colon\; \frac{8a_3\,\norm{1}_{L^2(h)}^2}{i^{4/3}\,\norm{1}_{L^6(h)}^2} +i^{2/3}\,\frac{4\pi\chi(M)}{3^{1/3}\norm{1}_{L^6(h)}^2}
\leq -i^{1/3}.
\]
We take $w=1$ and define $u_i,v_i$ as before. Using the Gauss--Bonnet theorem $\int_M\scal_h\,\diff\mu_h = 4\pi\chi(M)$, we obtain similarly as above (with $p=\frac{2\cdot3}{3-2} = 6$) for $i\geq i_h$:
\[ \begin{split}
E_g(v_i) &= \frac{\displaystyle a_3\norm{w}_{L^2(h)}^2 \,\norm{u_i'}_{L^2(\R)}^2
+\int_M\bigg(a_3\abs{\diff w}_h^2 +\scal_h\,w^2\bigg)\diff\mu_h \,\norm{u_i}_{L^2(\R)}^2
}{\norm{w}_{L^p(h)}^2 \, \norm{u_i}_{L^p(\R)}^2}\\
&\leq \frac{8a_3\,\norm{1}_{L^2(h)}^2}{i^{1+2/6}\,\norm{1}_{L^6(h)}^2} +i^{2/3}\,\frac{4\pi\chi(M)}{3^{1/3}\,\norm{1}_{L^6(h)}^2}
\leq -i^{1/3} .
\end{split} \]
Thus $\Yam_{M\times\R}(g) \leq \Yamfunc_g(v_i) \leq -i^{1/3}$ and $\Yam_{M\times \R}(h+\diff t^2) \leq \inf\set{-i^{1/3} \suchthat i\in\coi{i_h}{\infty}} = -\infty$.
\end{proof}

\begin{lemma} \label{balls}
Let $m,n\in\N_{\geq3}$, let $g_0$ be a Riemannian metric on the open $n$-ball $B^n$. Then there is a metric $g\in\Metr(B^n)$ with $\Yam_{B^n}(g)\leq -m$ which coincides with $g_0$ outside a compact subset $K$ of $B^n$.
\end{lemma}

\begin{proof}
Let $M$ be an $(n-1)$-dimensional compact submanifold of $B^n$ which admits a Riemannian metric $h$ of scalar curvature $-1$ (and hence $\Yam_M(h)<0$ if $n\geq4$, and $\chi(M)<0$ if $n=3$): If $n\geq4$, we can choose $M$ diffeomorphic to $S^{n-1}$; if $n=3$, we can choose $M$ diffeomorphic to a closed orientable surface of genus $2$. There exist a relatively compact (tubular) neighborhood $U$ of $M$ in $B^n$ and a diffeomorphism $\varphi\colon M\times\R\to U$. Let $i_h$ be as in Lemma \ref{cylinder} (with the $n$ there replaced by $n-1$). We choose a number $i\geq i_h$ with $i^{1/n}\geq m$, and a Riemannian metric $g$ on $B^n$ whose restriction to $\varphi(M\times[0,3i])$ is $(\varphi^{-1})^\ast(h+\diff t^2)$ and whose restriction to $B^n\without\bar{U}$ is $g_0$. This yields $\Yam_{B^n}(g) \leq \Yam_{M\times\R}(\varphi^\ast g) \leq -i^{1/n} \leq -m$ by Fact \ref{increasing} and Lemma \ref{cylinder}.
\end{proof}

\begin{proof}[Proof of Theorem \ref{imagetheorem}(\ref{imagetheoremsigma})]
From Gromov's h-principle \cite[Theorem 4.5.1]{Gromov1969} (which holds for manifolds each of whose connected components is noncompact) we know that there exists a metric $g_0\in \Metr(M)$ with positive scalar curvature. Clearly, $\Yam_M(g_0)\geq 0$. Hence $0\leq\sigma(M)\leq\sigma(S^n)$.

\smallskip
We choose an embedded open $n$-ball $B$ in $M$. For any $m\in\N_{\geq3}$, Lemma \ref{balls} gives us a metric $g_1$ on $M$ which coincides with $g_0$ outside a compact subset $K$ of $B$ and satisfies $\Yam_B(g_1)\leq -m$, hence also $\Yam_M(g_1)\leq-m$ by Fact \ref{increasing}. We consider the path $g\colon[0,1]\to\Metr(M)$ from $g_0$ to $g_1$ given by $g(t)\define(1-t)g_0+tg_1$. This path is continuous with respect to the fine $C^2$-topology on $\Metr(M)$: for every $t_0\in[0,1]$ and every neighborhood $U$ of $\graph(j^2(g(t_0)))$ in $J^2\Sym^2_+T^\ast M$, the set of $t\in[0,1]$ with $\graph(j^2(g(t)))\subseteq U$ is open in $[0,1]$ because all $g(t)$ coincide outside the compact subset $K$ of $M$.

\smallskip
According to Theorem \ref{finecontinuity}, the map $[0,1]\to\R\cup\set{-\infty}$ (which actually takes only values in $\R$) given by $t\mapsto\Yam_M(g(t))$ is therefore continuous, being a composition of continuous maps. Thus $\cci{-m}{\Yam_M(g_0)}$ is contained in the image of $\Yam_M$. Since this holds for every $m$, the interval $\oci{-\infty}{\Yam_M(g_0)}$ is contained in the image of $\Yam_M$.

\smallskip
It remains to show that there is also a metric $h\in\Metr(M)$ with $\Yam_M(h)=-\infty$. We choose a compact exhaustion $(K_i)_{i\in\N}$ of $M$ and a sequence of open balls $B_i\subset K_{i+1}\without K_i$; this is possible: since $M$ is noncompact, each $K_{i+1}\without K_i$ has nonempty interior. Lemma \ref{balls} yields for each $i\in\N$ a metric $h_i$ on $B_i$ which coincides outside a compact subset of $B_i$ with $g_0$ and satisfies $\Yam_{B_i}(h_i)\leq -i$. We define $h\in\Metr(M)$ by $h\restrict_{B_i}=h_i$ for every $i\in\N$, and $h=g_0$ on $M\without\bigcup_{i\in\N}B_i$. By Fact \ref{increasing}, we have $\Yam_M(h) \leq \Yam_{B_i}(h) = \Yam_{B_i}(h_i) \leq -i$ for all $i$; hence $\Yam_M(h)=-\infty$.
\end{proof}

\begin{proof}[Proof of Theorem \ref{imagetheorem}(\ref{imagetheorempositive})]
If the $n$-manifold $M$, each of whose connected components is noncompact, is diffeomorphic to an open subset of a closed $n$-manifold $\bar{M}$, then there exists an embedding $\iota\colon M\to\bar{M}$ such that for each connected component $C$ of $\bar{M}$ the set $C\without\iota(M)$ has nonempty interior. (Let $C$ be a connected component of $\bar{M}$ such that $M_C\define C\cap\iota(M)$ is nonempty. There exists a smooth embedding $\gamma\colon\cci{0}{1}\to C$ with $\gamma^{-1}(M_C)=\coi{0}{1}$. Moreover, there is a closed tubular neighborhood $A$ in $M_C$ of the image of $\gamma\restrict_{\coi{0}{1}}$ such that $M_C\without A$ is diffeomorphic to $M_C$. Taking $\iota\restrict_{\iota^{-1}(C)}$ to be the inclusion $\iota^{-1}(C)\cong M_C\cong M_C\without A\to C$ for each $C$, we obtain an embedding $\iota$ with the claimed property.) We choose such an embedding and identify $M$ with $\iota(M)$.

\smallskip
We extend the constant function $a_n$ on $M$ to a function $s\in C^\infty(\bar{M},\R)$ which is somewhere negative on each connected component of $\bar{M}$; this is possible because $C\without M$ has nonempty interior for each connected component $C$ of $\bar{M}$. By \cite[Theorem 1.1]{KazdanWarner}, $\bar{M}$ admits a Riemannian metric $\bar{g}$ with scalar curvature $s$. The metric $g\define\iota^\ast\bar{g}$ on $M$ has constant scalar curvature $a_n$.

\smallskip
Let $p=p_n$. By the Sobolev embedding theorem on $(\bar{M},\bar{g})$, there is a constant $c\in\R_{>0}$ such that $\norm{u}_{L^p(\bar{g})} \leq c\norm{u}_{H^{1,2}(\bar{g})}$ holds for all $u\in C^\infty(\bar{M},\R)$. Every test function $v\in C^\infty_c(M,\R_{\geq0})\without\set{0}$ can be extended by $0$ to a function $\bar{v}\in C^\infty(\bar{M},\R_{\geq0})\without\set{0}$ and thus satisfies
\[ \begin{split}
\norm{v}_{L^p(g)}^2 &= \norm{\bar{v}}_{L^p(\bar{g})}^2 \leq c^2\norm{\bar{v}}_{H^{1,2}(\bar{g})}^2 = c^2\norm{v}_{H^{1,2}(g)}^2
= \frac{c^2}{a_n}\int_M\Big(a_n\abs{\diff v}_g^2 +a_nv^2\Big)\diff\mu_g\\
&= \frac{c^2}{a_n}\Yamfunc_g(v)\norm{v}_{L^p(g)}^2 .
\end{split} \]
This yields $\Yamfunc_g(v)\geq a_n/c^2$ for all test functions $v$, hence $\Yam_M(g)\geq a_n/c^2 >0$ and $\sigma(M)>0$.
\end{proof}

\section{The compact-open discontinuity of the Yamabe map: proof of Theorem \ref{cocontinuity}}

\begin{proof}[Proof of Theorem \ref{cocontinuity}]
For the compact-open $C^2$-topology, upper semicontinuity of $\Yam_M$ and continuity at metrics $g$ with $\Yam_M(g)=-\infty$ have been proved in Lemma \ref{coupper}. If $M$ is compact, then the fine $C^2$-topology coincides with the compact-open $C^2$-topology, so Theorem \ref{finecontinuity} yields the compact-open $C^2$-continuity; of course this continuity was already known from \cite[Proposition 4.31]{Besse}. It remains to show that if $M$ is noncompact, then at each metric $g\in\Metr(M)$ with $\Yam_M(g)>-\infty$ the Yamabe map $\Yam_M$ is not (lower semi)continuous with respect to the compact-open $C^\infty$-topology (and hence neither with respect to any other compact-open $C^k$-topology).

\smallskip
Theorem \ref{imagetheorem}(\ref{imagetheoremsigma}) says that $M$ admits a metric $g_{-\infty}$ with $\Yam_M(g_{-\infty})=-\infty$. We choose a compact exhaustion $(K_i)_{i\geq0}$ of $M$ and define $(g_i)_{i\geq0}$ in $\Metr(M)$ by
\[
g_i \define
\begin{cases} g &\text{on $K_i$}\\
g_{-\infty} &\text{on $M\without K_{i+1}$}\\
\end{cases},
\]
and on $K_{i+1}\without K_i$ in an arbitrary way such that $g_i$ becomes a smooth metric on $M$. Then $(g_i)_{i\geq0}$ converges to $g$ in the compact-open $C^\infty$-topology: since every compact subset $K$ of $M$ is contained in some $K_j$, we have $\norm{g_i-g}_{C^r(K;g)} \leq \norm{g_i-g}_{C^r(K_j;g)} = 0$ for $i\geq j$ and all $r\in\N$.

\smallskip
By Theorem \ref{gtheorem}(\ref{gtheoreminftyb}), $\Yam_M(g_{-\infty})=-\infty$ implies $\lim_{i\to\infty}\Yam_{M\without K_i}(g_{-\infty}) = -\infty$. Since the sequence $\big(\Yam_{M\without K_i}(g_{-\infty})\big)_{i\geq0}$ is monotonically increasing, it must be constant $-\infty$.

\smallskip
Hence $\Yam_M(g_i) \leq \Yam_{M\without K_{i+1}}(g_i) = \Yam_{M\without K_{i+1}}(g_{-\infty}) = -\infty$ for all $i\geq0$. But the limit metric $g$ satisfies $\Yam_M(g)>-\infty$. This shows that $\Yam_M$ is not compact-open $C^\infty$-continuous at $g$.
\end{proof}

\section{Preparations for the uniform continuity proof} \label{uniformprep}

\begin{lemma} \label{uniformlemma}
Let $M$ be a manifold, let $\eps\in\R_{>0}$. Every $g\in\Metr(M)$ has a neighborhood $\mathcal{U}$ with respect to the uniform $C^2$-topology such that the following properties hold for all $h\in\mathcal{U}$:
\begin{enumerate}
\item\label{unione} $\forall \alpha\in T^\ast M\colon \bigabs{\abs{\alpha}_h^2 -\abs{\alpha}_g^2} \leq \eps\abs{\alpha}_g^2$.
\item\label{unitwo} $\Bigabs{\F{h}{g}-1} \leq \eps$.
\item\label{unithree} $\abs{\scal_h-\scal_g} \leq \tfrac{\eps}{2}\big(1+\abs{\Ric_g}_g\big)$.
\end{enumerate}
\end{lemma}

\noindent\emph{Remark.} The occurrence of $\abs{\Ric_g}_g$ on the right-hand side of \eqref{unithree} is not surprising, because the linearization of $g\mapsto\scal_g$ \cite[Theorem 1.174(e)]{Besse} involves the Ricci tensor. But in order to prove the lemma, also the remainder term in the Taylor expansion has to be estimated on arbitrary noncompact manifolds. Therefore it is hard to avoid the slightly tedious elementary arguments in the following proof.

\begin{proof}
Since $\R_{>0}\ni t\mapsto\frac{1}{t}$ is continuous, there exists a $\delta\in\R_{>0}$ such that $\bigabs{\frac{1}{t}-1}\leq\eps$ holds for all $t\in\R_{>0}$ with $\abs{t-1}\leq\delta$. We claim that every $h\in\mathcal{N}_{g,\delta,0}$ satisfies \eqref{unione}. In order to prove that, we consider $h\in\mathcal{N}_{g,\delta,0}$ and $x\in M$ and $\alpha\in T_x^\ast M$. The spectral theorem yields a $g$-orthonormal basis $(e_1,\dots,e_n)$ of $T_xM$ which is also $h$-orthogonal; thus there are $h_1,\dots,h_n\in\R_{>0}$ with $\forall i,j\colon h(e_i,e_j)=h_i\delta_{ij}$. The condition $h\in\mathcal{N}_{g,\delta,0}$ implies $\abs{h-g}_g(x)\leq\delta$, i.e.\ $\delta^2 \geq \sum_{i,j=1}^n(h(e_i,e_j)-g(e_i,e_j))^2 = \sum_{i=1}^n(h_i-1)^2$. In particular $\forall i\colon \abs{h_i-1}\leq\delta$, hence $\forall i\colon \bigabs{\frac{1}{h_i}-1}\leq\eps$. For the numbers $\alpha_i\define\alpha(e_i)$, we compute (using that $\big(h_1^{-1/2}e_1,\dots,h_n^{-1/2}e_n\big)$ is an $h$-orthonormal basis of $T_xM$):
\[
\Bigabs{\abs{\alpha}_h^2 -\abs{\alpha}_g^2}
= \Biggabs{\sum_{i=1}^n\frac{1}{h_i}\alpha_i^2 -\sum_{i=1}^n\alpha_i^2}
\leq \sum_{i=1}^n\biggabs{\frac{1}{h_i} -1}\alpha_i^2
\leq \eps\abs{\alpha}_g^2 .
\]
This proves our claim; in particular, every element $h$ of $\mathcal{U}_1\define \mathcal{N}_{g,\delta,2} \subseteq \mathcal{N}_{g,\delta,0}$ satisfies \eqref{unione}.

\smallskip
Since $(\R_{>0})^n\ni(t_1,\dots,t_n)\mapsto\prod_{i=1}^n\sqrt{t_i}$ is continuous, there exists a number $\delta_2\in\R_{>0}$ such that $\bigabs{\prod_{i=1}^n\sqrt{t_i} -1} \leq\eps$ holds for all $(t_1,\dots,t_n)\in(\R_{>0})^n$ with $\forall i\colon\abs{t_i-1} \leq\delta_2$. We claim that every $h\in\mathcal{N}_{g,\delta_2,0}$ satisfies \eqref{unitwo}. In order to prove that, we consider $h\in\mathcal{U}_2$ and $x\in M$ and again a $g$-orthonormal basis $(e_1,\dots,e_n)$ of $T_xM$ which is also $h$-orthogonal, and we define $h_1,\dots,h_n$ as before. We obtain
\[
\Bigabs{\F{h}{g}-1}(x) = \biggabs{\frac{\diff\mu_h(e_1,\dots,e_n)}{\diff\mu_g(e_1,\dots,e_n)} -1}
= \Biggabs{\frac{\diff\mu_h\big(h_1^{-1/2}e_1,\dots,h_n^{-1/2}e_n\big)}{\diff\mu_g(e_1,\dots,e_n)} \prod_{i=1}^nh_i^{1/2} -1}
= \Biggabs{\prod_{i=1}^nh_i^{1/2} -1} .
\]
Since $\forall i\colon \abs{h_i-1}\leq\delta_2$ holds by the same argument as above, we get $\bigabs{\F{h}{g}-1}(x) \leq \eps$. This is true for every $x\in M$, which proves our claim; in particular, every $h\in\mathcal{U}_2\define \mathcal{N}_{g,\delta_2,2} \subseteq \mathcal{N}_{g,\delta_2,0}$ satisfies \eqref{unitwo}.

\smallskip
There exists a (small) number $\delta_3\in\ooi{0}{1}$ with
\begin{align*}
\frac{n\delta_3}{1-\delta_3} &\leq \frac{\eps}{2} ,
&\frac{2n^2}{1-\delta_3}\bigg(\frac{3n\delta_3^2}{2(1-\delta_3)^2} +\frac{3\delta_3}{2(1-\delta_3)}\bigg)
+\frac{2n^3}{1-\delta_3}\bigg(\frac{3\delta_3}{2(1-\delta_3)}\bigg)^2 &\leq \frac{\eps}{2} .
\end{align*}
Let $\mathcal{U}_3\define\mathcal{N}_{g,\delta_3,2}$. We claim that \eqref{unithree} holds for every $h\in\mathcal{U}_3$.

\smallskip
Let $x\in M$. We choose a basis $(e_1,\dots,e_n)$ of $T_xM$ with the same properties as before and define $h_1,\dots,h_n$ in the same way. Existence of normal coordinates at $x$ \comment{***\cite[2.89]{GHL} ***}tells us that there are local coordinates $(x_1,\dots,x_n)$ around $x$ such that the corresponding coordinate vector fields $\partial_1,\dots,\partial_n$ satisfy at $x$ the equations $\partial_i(x)=e_i$ and $\Gamma_{ij}^k(x)=0$ for all $i,j,k\in\set{1,\dots,n}$, where $\Gamma_{ij}^k$ are the Christoffel symbols of the metric $g$ with respect to the local coordinates. As usual, $g^{ij}$ and $h^{ij}$ denote the elements of the inverses of the matrix-valued functions $(g_{ij})_{i,j=1\dots n}$ and $(h_{ij})_{i,j=1\dots n}$ given by $g_{ij}=g(\partial_i,\partial_j)$ and $h_{ij}=h(\partial_i,\partial_j)$. At $x$, they satisfy $g^{ij}(x) = \delta_{ij}$ and $h^{ij}(x)=\frac{1}{h_i}\delta_{ij}$. Let $\Tamma_{ij}^k$ denote the Christoffel symbols of $h$ with respect to our local coordinates. In the following, all sums run from $1$ to $n$.

\smallskip
From $h\in\mathcal{N}_{g,\delta_3,2}$, we deduce $\forall i\colon \delta_3\geq\abs{h_i-1}$ as before; moreover, with $\nabla$ denoting the Levi-Civita connection of $g$ and using $\nabla g=0$,
\begin{align*}
\delta_3^2 &\geq \abs{\nabla h}^2_g(x)
= \sum_{i,j,k}\Big((\nabla_{\partial_i}h)(\partial_j,\partial_k)\Big)^2(x)
= \sum_{i,j,k}\bigg(\partial_ih_{jk} -\sum_l\Gamma_{ij}^lh_{lk} -\sum_l\Gamma_{ik}^lh_{lj}\bigg)^2(x)\\
&= \sum_{i,j,k}(\partial_ih_{jk})^2(x)
\end{align*}
and
\[ \begin{split}
\delta_3^2 &\geq \abs{\nabla\nabla h}^2_g(x)\\
&= \sum_{i,j,k,l}\bigg(\partial_l(\nabla_{\partial_i}h)(\partial_j,\partial_k)\\[-1ex]
&\mspace{70mu}-\sum_m\bigg(\Gamma_{li}^m(\nabla_{\partial_m}h)(\partial_j,\partial_k) +\Gamma_{lj}^m(\nabla_{\partial_i}h)(\partial_m,\partial_k) +\Gamma_{lk}^m(\nabla_{\partial_i}h)(\partial_j,\partial_m)\bigg)\bigg)^2(x)\\
&= \sum_{i,j,k,l}\bigg(\partial_l\bigg(\partial_ih_{jk} -\sum_m\Gamma_{ij}^mh_{mk} -\sum_m\Gamma_{ik}^mh_{mj}\bigg)\bigg)^2(x)\\
&= \sum_{i,j,k,l}\bigg(\partial_l\partial_ih_{jk} -\sum_m(\partial_l\Gamma_{ij}^m)h_{mk} -\sum_m(\partial_l\Gamma_{ik}^m)h_{mj}\bigg)^2(x) .
\end{split} \]
Hence we obtain for all $i,j,k,l\in\set{1,\dots,n}$:
\begin{equation} \label{estimates} \begin{split}
\delta_3 &\geq \abs{h_i-1} ,\\
\delta_3 &\geq \abs{\partial_ih_{jk}}(x) ,\\
\delta_3 &\geq \Bigabs{ \partial_l\partial_ih_{jk} -(\partial_l\Gamma_{ij}^k)h_k -(\partial_l\Gamma_{ik}^j)h_j }(x) .
\end{split} \end{equation}

Recall that the Christoffel symbols are given by $\Tamma_{ab}^c = \frac{1}{2}\sum_mh^{cm}(\partial_ah_{bm} +\partial_bh_{am} -\partial_mh_{ab})$. Since every function $A\in C^\infty(\R^n,\GL(n))$ satisfies $\partial_i(A^{-1}) = -A^{-1}(\partial_iA)A^{-1}$, we get
\begin{equation} \label{Tamma} \begin{split}
(\partial_d\Tamma_{ab}^c)(x)
&= -\sum_m\bigg(\frac{\partial_dh_{cm}}{2h_ch_m}\Big(\partial_ah_{bm} +\partial_bh_{am} -\partial_mh_{ab}\Big)\bigg)(x)\\
&\mspace{20mu}+\frac{1}{2h_c}\Big(\partial_d\partial_ah_{bc} +\partial_d\partial_bh_{ac} -\partial_d\partial_ch_{ab}\Big)(x) .
\end{split} \end{equation}
Using the symmetry $\Gamma_{ij}^k=\Gamma_{ji}^k$, we obtain from \eqref{estimates}:
\begin{align*}
&\Bigabs{ \partial_d\partial_ah_{bc} +\partial_d\partial_bh_{ac} -\partial_d\partial_ch_{ab} -2h_c\partial_d\Gamma_{ab}^c }(x)\\
&\mspace{50mu}= \Big|
\big(\partial_d\partial_ah_{bc} -h_c\partial_d\Gamma_{ab}^c -h_b\partial_d\Gamma_{ac}^b \big)\\
&\mspace{50mu}\mspace{25mu}+\big(\partial_d\partial_bh_{ac} -h_c\partial_d\Gamma_{ba}^c -h_a\partial_d\Gamma_{bc}^a\big)
-\big(\partial_d\partial_ch_{ab} -h_b\partial_d\Gamma_{ca}^b -h_a\partial_d\Gamma_{cb}^a\big) \Big|(x)\\
&\mspace{50mu}\leq 3\delta_3 .
\end{align*}
Together with \eqref{Tamma} and \eqref{estimates}, this yields
\begin{align*}
\Bigabs{\partial_d\Tamma_{ab}^c -\partial_d\Gamma_{ab}^c}(x)
&\leq \sum_m\biggabs{\frac{\partial_dh_{cm}}{2h_ch_m}\Big(\partial_ah_{bm} +\partial_bh_{am} -\partial_mh_{ab}\Big)}(x)\\
&\mspace{20mu} +\biggabs{ \frac{\partial_d\partial_ah_{bc} +\partial_d\partial_bh_{ac} -\partial_d\partial_ch_{ab} -2h_c\partial_d\Gamma_{ab}^c}{2h_c} }(x)\\
&\leq \sum_m\frac{3\delta_3^2}{2h_ch_m} +\frac{3\delta_3}{2h_c}\\
&\leq \frac{3n\delta_3^2}{2(1-\delta_3)^2} +\frac{3\delta_3}{2(1-\delta_3)} .
\end{align*}
The well-known local coordinate formula for scalar curvature tells us that
\[
\scal_g(x) = \sum_{a,b,c}g^{ab}\bigg(\partial_c\Gamma_{ab}^c -\partial_b\Gamma_{ac}^c +\sum_d\Gamma_{ab}^d\Gamma_{cd}^c -\sum_d\Gamma_{ac}^d\Gamma_{bd}^c\bigg)(x)
= \sum_{a,c}\big(\partial_c\Gamma_{aa}^c -\partial_a\Gamma_{ac}^c\big)(x)
\]
and
\[ \begin{split}
\scal_h(x) &= \sum_{a,b,c}h^{ab}\bigg(\partial_c\Tamma_{ab}^c -\partial_b\Tamma_{ac}^c +\sum_d\Tamma_{ab}^d\Tamma_{cd}^c -\sum_d\Tamma_{ac}^d\Tamma_{bd}^c\bigg)(x)\\
&= \sum_{a,c}\frac{\partial_c\Tamma_{aa}^c -\partial_a\Tamma_{ac}^c +\sum_d\Tamma_{aa}^d\Tamma_{cd}^c -\sum_d\Tamma_{ac}^d\Tamma_{ad}^c}{h_a}(x) .
\end{split} \]
The local coordinate formula for $\Ric_g$ yields for each $a\in\set{1,\dots,n}$:
\[ \begin{split}
\abs{\Ric_g}_g &\geq \abs{\Ric_g(e_a,e_a)}
= \biggabs{ \sum_c\bigg(\partial_c\Gamma_{aa}^c -\partial_a\Gamma_{ac}^c +\sum_d\Gamma_{aa}^d\Gamma_{cd}^c -\sum_d\Gamma_{ac}^d\Gamma_{ad}^c\bigg) }(x)\\
&= \biggabs{ \sum_c\big(\partial_c\Gamma_{aa}^c -\partial_a\Gamma_{ac}^c\big) }(x) \,.
\end{split} \]
Using the estimate (which follows from \eqref{estimates})
\[ \begin{split}
\bigabs{\Tamma_{ab}^c}(x)
&\leq \frac{1}{2}\sum_m\Bigabs{ h^{cm}(\partial_ah_{bm} +\partial_bh_{am} -\partial_mh_{ab}) }(x)
\leq \frac{3\delta_3}{2(1-\delta_3)} ,
\end{split} \]
we obtain finally:
\[ \begin{split}
\bigabs{\scal_h-\scal_g}(x)
&= \Biggabs{ \sum_{a,c}\frac{\partial_c\Tamma_{aa}^c -\partial_a\Tamma_{ac}^c +\sum_d\Tamma_{aa}^d\Tamma_{cd}^c -\sum_d\Tamma_{ac}^d\Tamma_{ad}^c
-h_a\big(\partial_c\Gamma_{aa}^c -\partial_a\Gamma_{ac}^c\big)}{h_a} }(x)\\
&\leq \sum_{a,c}\Biggabs{ \frac{\big(\partial_c\Tamma_{aa}^c -\partial_c\Gamma_{aa}^c\big) -(\partial_a\Tamma_{ac}^c -\partial_a\Gamma_{ac}^c)}{h_a} }(x)\\
&\mspace{20mu}+\sum_{a,c,d}\Biggabs{ \frac{\Tamma_{aa}^d\Tamma_{cd}^c -\Tamma_{ac}^d\Tamma_{ad}^c }{h_a} }(x)
+\sum_a\frac{\abs{1-h_a}}{h_a}\Biggabs{\sum_c\big(\partial_c\Gamma_{aa}^c -\partial_a\Gamma_{ac}^c\big) }(x)\\
&\leq \frac{2n^2}{1-\delta_3}\bigg(\frac{3n\delta_3^2}{2(1-\delta_3)^2} +\frac{3\delta_3}{2(1-\delta_3)}\bigg)
+\frac{2n^3}{1-\delta_3}\bigg(\frac{3\delta_3}{2(1-\delta_3)}\bigg)^2
+\frac{n\delta_3}{1-\delta_3}\abs{\Ric_g}_g(x)\\
&\leq \frac{\eps}{2}\big(1+\abs{\Ric_g}_g(x)\big) .
\end{split} \]
This is true for every $x\in M$, which proves our claim.

\smallskip
The uniform $C^2$-neighborhood $\mathcal{U}\define \mathcal{U}_1\cap\mathcal{U}_2\cap\mathcal{U}_3$ of $g$ has the desired property.
\end{proof}

\begin{corollary} \label{uniformlemmamod}
Let $M$ be a manifold, let $\eps\in\R_{>0}$. If $g\in\Metr(M)$ admits a constant $c\in\R_{>0}$ with $\abs{\Ric_g}_g\leq c(1+\abs{\scal}_g)$, then it has a neighborhood $\mathcal{U}$ with respect to the uniform $C^2$-topology such that the following properties hold for all $h\in\mathcal{U}$:
\begin{enumerate}
\item $\forall \alpha\in T^\ast M\colon \bigabs{\abs{\alpha}_h^2 -\abs{\alpha}_g^2} \leq \eps\abs{\alpha}_g^2$.
\item $\Bigabs{\F{h}{g}-1} \leq \eps$.
\item $\abs{\scal_h-\scal_g} \leq \tfrac{\eps}{2}\big(1+\abs{\scal_g}\big)$.
\end{enumerate}
\end{corollary}
\begin{proof}
We apply Lemma \ref{uniformlemma} to $\tilde{\eps}\define\frac{\eps}{1+c}$ instead of $\eps$. Let $h$ be an element of the resulting uniform $C^2$-neighborhood $\mathcal{U}$ of $g$. Since $\tilde{\eps}\leq\eps$, we obtain our properties (1), (2) from the properties (1), (2) of \ref{uniformlemma}. Moreover, $\abs{\scal_h-\scal_g} \leq \frac{\tilde{\eps}}{2}\big(1+\abs{\Ric_g}_g\big) \leq \frac{\tilde{\eps}}{2}\big(1+c+c\abs{\scal_g}\big) \leq \tfrac{\eps}{2}\big(1+\abs{\scal_g}\big)$.
\end{proof}

\section{Continuity with respect to the uniform topology: proof of Theorem \ref{uniformcontinuity}}

\begin{proof}[Proof of Theorem \ref{uniformcontinuity}]
Upper semicontinuity and continuity at metrics with Yamabe constant $-\infty$ follow from Corollary \ref{upper}. It remains to prove lower semicontinuity at each metric $g\in\Metr(M)$ with $\Yam_M(g)>-\infty$ for which there exist constants $\delta,c\in\R_{>0}$ with $\norm{(\scal_g-\delta)_-}_{L^{n/2}(g)} < \infty$ and $\abs{\Ric_g}_g\leq c(1+\abs{\scal_g})$.

\smallskip
We start with the case $\delta=1$. Let $\eps_0\in\ooi{0}{1}$. Let $p=p_n$. There exists a (small) $\eps\in\ooi{0}{1}$ such that
\begin{equation*} \begin{split}
\tfrac{3}{2}\eps(7-\eps)(1-\eps)^{-2/p}\,\bignorm{(\scal_g-1)_-}_{L^{n/2}(g)} &\leq \eps_0 \\
\text{and}\quad \bigg(1-\frac{(1-\eps)^2}{(1+\eps)^{2/p}}\bigg)\abs{\Yam_M(g)} &\leq \eps_0 .
\end{split} \end{equation*}
Let
\begin{align*}
A &\define \Bigset{x\in M \Bigsuchthat \tfrac{\eps}{2}\big(1-\tfrac{\eps}{2}\big)^{-1} \leq \scal_g(x)} ,\\
B &\define \Bigset{x\in M \Bigsuchthat 0 < \scal_g(x) < \tfrac{\eps}{2}\big(1-\tfrac{\eps}{2}\big)^{-1}} ,\\
C &\define \Bigset{x\in M \Bigsuchthat \scal_g(x) \leq 0} .
\end{align*}

\smallskip
We choose a uniform $C^2$-neighborhood of $\mathcal{U}$ with the properties stated in Corollary \ref{uniformlemmamod}. For every $h\in\mathcal{U}$ and $v\in C^\infty_c(M,\R_{\geq0})$, we have
\[ \begin{split}
&-(1-\eps)^2\int_A\scal_g\,v^2\,\diff\mu_g +\int_A\Big(\scal_g-\tfrac{\eps}{2}\abs{\scal_g}-\tfrac{\eps}{2}\Big)\,v^2\,\F{h}{g}\,\diff\mu_g\\
&\geq -(1-\eps)^2\int_A\scal_g\,v^2\,\diff\mu_g +(1-\eps)\int_A\Big(\big(1-\tfrac{\eps}{2}\big)\scal_g-\tfrac{\eps}{2}\Big)\,v^2\,\diff\mu_g\\
&= \tfrac{\eps}{2}(1-\eps)\int_A(\scal_g-1)\,v^2\,\diff\mu_g\\
&\geq -\tfrac{\eps}{2}(1-\eps)\int_A(\scal_g-1)_-\,v^2\,\diff\mu_g
\;\geq\; -\tfrac{\eps}{2}(1-\eps)\,\bignorm{(\scal_g-1)_-}_{L^{n/2}(A;g)}\,\norm{v}_{L^p(A;g)}^2\\
&\geq -\tfrac{\eps}{2}(1-\eps)\,\bignorm{(\scal_g-1)_-}_{L^{n/2}(g)}\,\norm{v}_{L^p(g)}^2
\end{split} \]
and
\[ \begin{split}
&-(1-\eps)^2\int_B\scal_g\,v^2\,\diff\mu_g +\int_B\Big(\scal_g-\tfrac{\eps}{2}\abs{\scal_g}-\tfrac{\eps}{2}\Big)\,v^2\,\F{h}{g}\,\diff\mu_g\\
&\geq -(1-\eps)^2\int_B\scal_g\,v^2\,\diff\mu_g +(1+\eps)\int_B\Big(\big(1-\tfrac{\eps}{2}\big)\scal_g-\tfrac{\eps}{2}\Big)\,v^2\,\diff\mu_g\\
&\geq \tfrac{\eps}{2}(5-3\eps)\int_B\scal_g\,v^2\,\diff\mu_g -\tfrac{\eps}{2}(1+\eps)\int_Bv^2\,\diff\mu_g
\;\geq\; \tfrac{\eps}{2}(5-3\eps)\int_B(\scal_g-1)\,v^2\,\diff\mu_g\\
&\geq -\tfrac{\eps}{2}(5-3\eps)\,\bignorm{(\scal_g-1)_-}_{L^{n/2}(g)}\,\norm{v}_{L^p(g)}^2
\end{split} \]
and
\[ \begin{split}
&-(1-\eps)^2\int_C\scal_g\,v^2\,\diff\mu_g +\int_C\Big(\scal_g-\tfrac{\eps}{2}\abs{\scal_g}-\tfrac{\eps}{2}\Big)\,v^2\,\F{h}{g}\,\diff\mu_g\\
&\geq -(1-\eps)^2\int_C\scal_g\,v^2\,\diff\mu_g
+(1+\eps)\int_C\Big(\big(1+\tfrac{\eps}{2}\big)\scal_g -\tfrac{\eps}{2}\Big)\,v^2\,\diff\mu_g\\
&= \tfrac{\eps}{2}(7-\eps)\int_C\scal_g\,v^2\,\diff\mu_g -\tfrac{\eps}{2}(1+\eps)\int_Cv^2\,\diff\mu_g
\;\geq\; \tfrac{\eps}{2}(7-\eps)\int_C(\scal_g-1)\,v^2\,\diff\mu_g\\
&\geq -\tfrac{\eps}{2}(7-\eps)\,\bignorm{(\scal_g-1)_-}_{L^{n/2}(g)}\,\norm{v}_{L^p(g)}^2 \,;
\end{split} \]
hence, using $\int_Mv^p\,\diff\mu_g = \int_Mv^p\big(\F{h}{g}\big)^{-1}\diff\mu_h \leq (1-\eps)^{-1}\int_Mv^p\,\diff\mu_h$:
\[ \begin{split}
&-(1-\eps)^2\int_M\scal_g\,v^2\,\diff\mu_g +\int_M\Big(\scal_g-\tfrac{\eps}{2}\abs{\scal_g}-\tfrac{\eps}{2}\Big)\,v^2\,\F{h}{g}\,\diff\mu_g\\
&\geq -\tfrac{3}{2}\eps(7-\eps)\,\bignorm{(\scal_g-1)_-}_{L^{n/2}(g)}\,\norm{v}_{L^p(g)}^2\\
&\geq -\tfrac{3}{2}\eps(7-\eps)(1-\eps)^{-2/p}\,\bignorm{(\scal_g-1)_-}_{L^{n/2}(g)}\,\norm{v}_{L^p(h)}^2 \,.
\end{split} \]

We get for all $h\in\mathcal{U}$ and $v\in C_c^\infty(M,\R_{\geq0})$ with $\norm{v}_{L^p(h)}=1$:
\[ \begin{split}
\Yamfunc_h(v)
&= a_n\int_M\abs{\diff v}_h^2\,\F{h}{g}\,\diff\mu_g +\int_M\scal_h\,v^2\,\F{h}{g}\,\diff\mu_g\\
&\geq (1-\eps)^2\int_M a_n\,\abs{\diff v}_g^2\,\diff\mu_g +\int_M\Big(\scal_g-\tfrac{\eps}{2}\abs{\scal_g}-\tfrac{\eps}{2}\Big)\,v^2\,\F{h}{g}\,\diff\mu_g\\
&= (1-\eps)^2 \bigg(\Yamfunc_g(v)\,\norm{v}_{L^p(g)}^2 -\int_M\scal_g\,v^2\,\diff\mu_g \bigg)
+\int_M\Big(\scal_g-\tfrac{\eps}{2}\abs{\scal_g}-\tfrac{\eps}{2}\Big)\,v^2\,\F{h}{g}\,\diff\mu_g\\
&\geq (1-\eps)^2\,\Yam_M(g)\,\norm{v}_{L^p(g)}^2 -\tfrac{3}{2}\eps(7-\eps)(1-\eps)^{-2/p}\,\bignorm{(\scal_g-1)_-}_{L^{n/2}(g)} \,.
\end{split} \]

Since $(1+\eps)^{-2/p} = (1+\eps)^{-2/p}\,\norm{v}_{L^p(h)}^2 \leq \norm{v}_{L^p(g)}^2 \leq (1-\eps)^{-2/p}\,\norm{v}_{L^p(h)}^2 = (1-\eps)^{-2/p}$, we obtain in the case $\Yam_M(g)\geq0$:
\[
(1-\eps)^2\,\Yam_M(g)\,\norm{v}_{L^p(g)}^2
\geq \frac{(1-\eps)^2}{(1+\eps)^{2/p}}\,\Yam_M(g)
= \Yam_M(g) -\bigg(1-\frac{(1-\eps)^2}{(1+\eps)^{2/p}}\bigg)\abs{\Yam_M(g)}
\geq \Yam_M(g) -\eps_0 ;
\]
and in the case $\Yam_M(g)<0$:
\[
(1-\eps)^2\,\Yam_M(g)\,\norm{v}_{L^p(g)}^2
\geq \frac{(1-\eps)^2}{(1-\eps)^{2/p}}\,\Yam_M(g)
\geq \Yam_M(g) ,
\]
because $2-2/p>0$. This yields in each case:
\[ \begin{split}
\Yamfunc_h(v)
&\geq \Yam_M(g) -2\eps_0 \,,
\end{split} \]
hence $\Yam_M(h)\geq \Yam_M(g)-2\eps_0$. Since there exists for every $\eps_0\in\R_{>0}$ a uniform $C^2$-neighborhood $\mathcal{U}$ such that this holds for all $h\in\mathcal{U}$, the Yamabe map is indeed lower semicontinuous in the case $\delta=1$.

\smallskip
Now we consider an arbitrary $\delta\in\R_{>0}$. Because of our assumption on $g$, the metric $\bar{g} = \delta g$ satisfies $\abs{\Ric_{\bar{g}}}_{\bar{g}} = \frac{1}{\delta}\abs{\Ric_g}_g \leq \frac{c}{\delta}(1+\abs{\scal_g}) = \frac{c}{\delta}(1+\delta\abs{\scal_{\bar{g}}}) \leq \tilde{c}(1+\abs{\scal_{\bar{g}}})$ for $\tilde{c}\define c\max\bigset{1,\frac{1}{\delta}}$, and
\[
\int_M\big((\scal_{\bar{g}}-1)_-\big)^{n/2}\diff\mu_{\bar{g}}
= \int_M\Big(\tfrac{1}{\delta}(\scal_g-\delta)_-\Big)^{n/2}\,\delta^{n/2}\,\diff\mu_g
= \int_M\big((\scal_g-\delta)_-\big)^{n/2}\diff\mu_g < \infty .
\]
Thus the case we have proved already (applied to $\bar{g},\tilde{c}$ instead of $g,c$) yields lower semicontinuity of $\Yam_M$ at $\bar{g}$ and hence, by conformal invariance of $\Yam_M$, also at $g$.
\end{proof}

\section{Discontinuity with respect to the uniform topology: proof of Example \ref{uniformexample}}

\begin{proof}[Proof of \ref{uniformexample}]
Since $\sigma(N)>0$, there exists a metric $h\in\Metr(N)$ with $\Yam_N(h)\geq0$. Like every nonempty closed manifold of dimension $\geq3$, $N$ admits a metric $h'$ with $\Yam_N(h')<0$. We choose a smooth path $(h_t)_{t\in[0,1]}$ in $\Metr(N)$ with $h_0=h'$ and $h_1=h$. Let $t_0\define\min\set{t\in[0,1] \suchthat \Yam_N(h_t)=0}$ (the minimum exists because $\Yam_N$ is continuous). By the solution of the Yamabe problem for closed manifolds, the conformal class of $h_{t_0}$ contains a metric $k_{t_0} = f^2h_{t_0}$ with scalar curvature $0$. For $t\in[0,1]$, let $k_t\define f^2h_t$ and $g_t\define k_t+\diff t^2$. Then $\Yam_N(k_t)=\Yam_N(h_t)<0$ for all $t<t_0$. Thus, for all $t<t_0$, Lemma \ref{cylinder} implies $\Yam_M(g_t)=-\infty$. On the other hand, $\Yam_M(g_{t_0})\geq 0$ because $\scal_{g_{t_0}}=0$. Every uniform $C^\infty$-neighborhood of $g_{t_0}$ contains metrics $g_t$ with $t<t_0$, because for each $r\in\N$, \,$\norm{g_t-g_{t_0}}_{C^r(g_{t_0})} = \norm{k_t-k_{t_0}}_{C^r(k_{t_0})} = \norm{f^2(h_t-h_{t_0})}_{C^r(k_{t_0})}$ tends to $0$ as $t\to t_0$. Hence $\Yam_M$ is not continuous at $g_{t_0}$ with respect to the uniform $C^\infty$-topology, and thus not continuous with respect to any uniform $C^k$-topology.
\end{proof}


\end{document}